\documentclass[12pt]{article}
 
\usepackage[margin=1in]{geometry} 
\usepackage{amsmath,amsthm,amssymb,enumerate}

\usepackage{xcolor}
\usepackage{hyperref}
\usepackage{graphicx}
\hypersetup{
    colorlinks,
    linkcolor={blue!90!black},
    citecolor={blue!100!black},
    urlcolor={blue!80!black}
}

\DeclareRobustCommand{\bl}[1]{{\color{blue}#1}}
\DeclareRobustCommand{\gr}[1]{{\color{green}#1}}
\DeclareRobustCommand{\re}[1]{{\color{red}#1}}

\usepackage{comment}
\usepackage{subcaption}
\usepackage{enumitem}

\usepackage{authblk}
\title{Asymptotic Theory for Graphical SLOPE: Precision Estimation and Pattern Convergence\footnote{The opinions expressed in this article are those of the authors and do not necessarily reflect the views of La Francaise Systematic Asset Management GmbH or any of its affiliates.}}
\author[1]{Ivan Hejný}
\author[2]{Giovanni Bonaccolto}
\author[3, 4]{Philipp Kremer}
\author[4]{Sandra Paterlini}
\author[5]{Małgorzata Bogdan}
\author[1]{Jonas Wallin}
\affil[1]{Department of Statistics, Lund University, Sweden}
\affil[2]{Department of Economics and Law, ``Kore'' University of Enna, Italy}
\affil[3]{La Francaise Systematic Asset Management GmbH, 
Frankfurt am Main, Germany}
\affil[4]{Department of Economics and Management, 
University of Trento, Italy}
\affil[5]{Institute of Mathematics, University of Wroclaw, Poland}
\date{}

\usepackage{titlesec}
\titleformat*{\section}{\large\bfseries}
\titleformat*{\subsection}{\small\bfseries}
\titleformat*{\subsubsection}{\small\bfseries}

\theoremstyle{plain}
\newtheorem{theorem}{Theorem}[section]
\newtheorem{proposition}[theorem]{Proposition}
\newtheorem{lemma}[theorem]{Lemma}

\theoremstyle{definition}

\newtheorem{example}[theorem]{Example}
\newtheorem{remark}[theorem]{Remark}

\newtheorem*{notation}{Notation}

\renewcommand{\vec}{\operatorname{vec}}   % override arrow version
\DeclareMathOperator{\vech}{vech}

\DeclareMathOperator{\Cov}{Cov}

\newcommand{\indep}{\perp\!\!\!\perp}

\begin{document}

\maketitle

\begin{abstract}
This paper studies Graphical SLOPE for precision matrix estimation, with emphasis on its ability to recover both sparsity and clusters of edges with equal or similar strength. In a fixed-dimensional regime, we establish that the root-$n$ scaled estimation error converges to the unique minimizer of a strictly convex optimization problem defined through the directional derivative of the SLOPE penalty. We also establish convergence of the induced SLOPE pattern, thereby obtaining an asymptotic characterization of the clustering structure selected by the estimator. A comparison with GLASSO shows that the grouping property of SLOPE can substantially improve estimation accuracy when the precision matrix exhibits structured edge patterns.

To assess the effect of departures from Gaussianity, we then analyze Gaussian-loss precision matrix estimation under elliptical distributions. In this setting, we derive the limiting distribution and quantify the inflation in variability induced by heavy tails relative to the Gaussian benchmark. We also study TSLOPE, based on the multivariate $t$-loss, and derive its limiting distribution. The results show that TSLOPE offers clear advantages over GSLOPE under heavy-tailed data-generating mechanisms. Simulation evidence suggests that these qualitative conclusions persist in high-dimensional settings, and an empirical application shows that SLOPE-based estimators, especially TSLOPE, can uncover economically meaningful clustered dependence structures.
\end{abstract}
%\begin{abstract}
%We develop an asymptotic theory for Graphical SLOPE in precision matrix estimation, with particular emphasis on its ability to recover not only sparsity but also clusters of edges with equal or similar strength. In a fixed-dimensional regime, we show that the root-$n$ scaled estimation error converges to the unique minimizer of a strictly convex problem involving the directional derivative of the SLOPE penalty. We also establish convergence of the induced SLOPE pattern, yielding an asymptotic characterization of the clustering structure recovered by the estimator. A comparison of the asymptotic distributions of GSLOPE and GLASSO shows that the clustering property of SLOPE can substantially improve estimation accuracy for structured precision matrices.

%We then specialize the general theory to Gaussian-loss precision matrix estimation under elliptical distributions, deriving the corresponding limiting distribution and quantifying the increase in variability caused by heavy tails relative to the Gaussian benchmark. We also analyze TSLOPE, based on the multivariate $t$ loss, and derive its limiting distribution. Our results highlight the advantage of TSLOPE over GSLOPE when the data-generating distribution is heavy-tailed. Additional simulations indicate that these qualitative findings extend to the high-dimensional setting. An empirical application further demonstrates that SLOPE-based estimators, especially TSLOPE, can uncover economically meaningful clustered dependence structures.
%\end{abstract}

\section{Introduction}
Graphical models provide a parsimonious representation of conditional
dependence among the components of a random vector $X=(X_1,\dots,X_p)^\top$.
For Gaussian graphical models (GGMs) $X\sim\mathcal N_p(0,\Sigma)$, conditional independencies are encoded by
zeros in the precision matrix $\Theta_0 := \Sigma^{-1}$:
%\textcolor{blue}{Gaussian graphical models (GGMs) provide a parsimonious representation of conditional
%dependence among the components of a random vector $X=(X_1,\dots,X_p)^\top$.
%In the Gaussian case $X\sim\mathcal N_p(0,\Sigma)$, conditional independencies are encoded by
%zeros in the precision matrix $\Theta_0 := \Sigma^{-1}$:}
$(\Theta_0)_{ij}=0$ if and only if $X_i \indep X_j \mid X_{-(i,j)}$
\cite{lauritzen1996graphical}. Estimating $\Theta_0$ from $n$ i.i.d.\ observations is therefore a central
task in network learning and multivariate analysis, with applications in genomics,
neuroscience, and financial econometrics.

A widely used approach estimates $\Theta_0$ by solving a convex penalized likelihood problem
over the cone $\mathrm{Sym}_+(p)$.
In particular, the graphical Lasso (GLasso) augments the Gaussian negative log-likelihood with
an entrywise $\ell_1$ penalty on off-diagonal elements, yielding a sparse precision estimate and
a sparse conditional independence graph \cite{friedman2008sparse}.
A large literature establishes statistical guarantees for GLasso and related estimators in
high-dimensional regimes under sparsity assumptions; see, e.g.,
\cite{RavikumarWainwrightRaskuttiYu2011} and references therein.

In many applications, however, the goal is not only to identify \emph{which} edges are present, but also to recover \emph{groups of edges with comparable or identical strength}. Such grouping arises in block-structured dependence, latent factor models, and hub-like networks, and is closely related to the equality-constrained structures studied in colored graphical models (see, e.g., \cite{HojsgaardLauritzen2008}), where edges sharing the same ``color'' are constrained to have identical parameters.

In these settings, an estimator may achieve a small global error yet still fragment truly homogeneous groups into slightly different magnitudes, thereby complicating interpretation. This motivates distinguishing overall estimation accuracy from the ability to recover the underlying \emph{clustering structure} implied by $\Theta_0$.

Graphical SLOPE (GSLOPE) is designed to target this second objective. It replaces the entrywise
$\ell_1$ penalty by a sorted $\ell_1$ (SLOPE) penalty on the off-diagonal precision entries,
thereby encouraging sparsity while allowing multiple edges to be tied to the same magnitude.
To formalize and assess this behavior, we use the notion of \emph{patterns} for polyhedral
regularizers \cite{bogdan2022pattern,graczyk2025graphical,hejny2025unveiling, hejny2025asymptotic},
which captures both sparsity and equality relations among nonzero coefficients.
This viewpoint suggests a natural empirical diagnostic: in addition to the usual RMSE, we
track a clustering RMSE that measures deviation from the clustering structure implied by the
target precision matrix. The resulting RMSE versus clustering-RMSE curves quantify the
trade-off induced by tuning: stronger regularization can improve clustering at the cost of
increased global error, and vice versa.

The paper makes three contributions.
First, in a fixed-$p$ asymptotic regime ($n\to\infty$ with $p$ fixed), we characterize the
limiting distribution of $\sqrt{n}\,(\widehat{\Theta}_n-\Theta_0)$ for GSLOPE under a general
smooth loss. The limit is given as the unique minimizer of a strictly convex quadratic-plus-penalty
problem involving the population Hessian and the directional derivative of the SLOPE penalty,
and we also establish convergence of the induced SLOPE pattern.

Second, we specialize the limit to Gaussian-loss precision estimation under elliptical models
and derive explicit expressions for the Hessian and the score covariance, thereby quantifying
variance inflation under heavy tails. We further analyze TSLOPE, based on the multivariate $t$
loss, and obtain the corresponding limiting quantities under a condition on the radial
distribution.

Third, we complement the theory with simulations that (i) validate the asymptotic RMSE
approximation, (ii) decompose error into a component compatible with the target pattern and a
clustering-error component, and (iii) compare GSLOPE and TSLOPE under $t$-distributed data
across degrees of freedom.

\subsection{Organization of the paper}
The paper is organized as follows. Section~\ref{sec:related} reviews the related literature. Section~\ref{sec:theoretical results} presents the main theoretical results, describing a general asymptotic framework for graphical SLOPE (Theorem~\ref{main conditions theorem}). Within this section, we apply the framework to establish the exact asymptotic error distribution for GSLOPE under elliptically distributed data (Theorem~\ref{theorem convergence in distribution} and Proposition~\ref{proposition for elliptical distributions}) and derive a similar asymptotic formula for TSLOPE using a multivariate $t$-distribution penalty (Theorem~\ref{theorem convergence in distribution for t distribution}). Section~\ref{sec:simulations} includes simulation studies that illustrate the asymptotic framework, highlighting TSLOPE's advantage in estimation and pattern recovery under heavy-tailed, block-structured data relative to GSLOPE, GLASSO, and TLASSO. Section~\ref{sec:empirical} applies these estimators to empirical financial data to identify economically meaningful groupings among stock portfolios. Section~\ref{sec:conclusion} concludes the paper. Auxiliary results, proofs, and figures are provided in Appendix~\ref{sec:appendix}.

\section{Related literature}\label{sec:related}
Sparse precision matrix estimation and Gaussian graphical modeling have been studied extensively over the last two decades. In Gaussian graphical models, conditional
independencies correspond to zeros in the precision matrix, which motivates convex
estimators that combine the Gaussian log-likelihood with a sparsity-inducing penalty; the
graphical Lasso is the canonical example \cite{friedman2008sparse}. A large body of
work establishes statistical guarantees for such estimators in high-dimensional regimes under
sparsity and regularity conditions, including rates and support recovery results for
$\ell_1$-penalized log-determinant programs \cite{RavikumarWainwrightRaskuttiYu2011}. Closely
related approaches include neighborhood selection based on nodewise Lasso regressions
\cite{MeinshausenBuhlmann2006} and $\ell_1$-constrained precision estimation such as CLIME
\cite{CaiLiuLuo2011}. These contributions collectively provide the methodological and
theoretical baseline against which more structured regularizers for graphical models can be
compared.

Beyond entrywise sparsity, a growing literature develops convex regularizers that encode
additional structure---such as grouping, fusion, hierarchical constraints, and hub patterns---in
graphical model estimation. From a geometric perspective, many of these regularizers are
\emph{polyhedral} norms or can be expressed via \emph{atomic} norms, enabling unified analyses
of estimation and model selection in both low- and high-dimensional settings. Recent work
connects graphical Lasso-type estimators to a broader family of atomic-norm formulations and
studies their high-dimensional \emph{pattern recovery} properties, emphasizing the role of
non-differentiability and the induced low-dimensional structure of the solution
\cite{graczyk2025graphical}. This line of research motivates penalties that recover
structured dependence patterns, not merely sparse edge sets.

In regression, the sorted $\ell_1$ penalty (SLOPE) was introduced as a convex procedure that
adapts to unknown sparsity and, with suitable tuning, can control the false discovery rate
\cite{BogdanVanDenBergSabattiSuCandes2015,BenjaminiHochberg1995}. Independently of its multiple
testing interpretation, SLOPE (and the related class of ordered weighted $\ell_1$ norms)
exhibits an important geometric property: by penalizing ordered absolute values with a
decreasing weight sequence, it naturally induces \emph{clustering} of coefficients, producing
groups of equal magnitudes. Recent work formalizes this phenomenon through the concept of
\emph{patterns} and provides guarantees for \emph{pattern recovery} under SLOPE, highlighting
how polyhedral geometry governs the equality and ordering relations recovered by the
estimator \cite{bogdan2022pattern}. These developments provide the conceptual foundation for transporting
SLOPE to matrix-valued problems such as precision matrix estimation, where clustering
corresponds to equality of edge strengths, closely related to the equality-constrained structures
studied in colored graphical models \cite{HojsgaardLauritzen2008}.

The present paper also relates to the broader literature on low-dimensional asymptotics for
estimators defined by non-smooth convex regularization. Classical results for Lasso-type
estimators in the fixed-$p$ regime characterize the limit distribution of the rescaled error
via an auxiliary convex optimization problem obtained from a local quadratic approximation of
the loss and the directional derivative of the penalty \cite{fu2000asymptotics,VanDerVaart1998}.
Recent work generalizes this program to a wide class of loss functions and polyhedral
regularizers, and further develops asymptotic results for \emph{pattern-valued} functionals,
establishing convergence of the induced patterns and providing tractable descriptions of
their limiting behavior \cite{hejny2025unveiling, hejny2025asymptotic}. Our analysis of
graphical SLOPE can be viewed as a concrete specialization of this framework to precision
matrix estimation, with a focus on how clustering patterns arise in the estimation error.

Finally, our heavy-tail analysis connects to robust graphical modeling under distributional
misspecification. In many empirical applications, Gaussianity is an imperfect approximation
and tail heaviness can inflate variability when Gaussian likelihood losses are employed.
Elliptical models provide a tractable generalization that preserves a form of radial symmetry
while allowing for heavy tails, and the multivariate $t$ distribution is a standard benchmark
within this class. Likelihood-based procedures tailored to $t$ models can therefore serve as
robust alternatives to Gaussian-loss estimators in heavy-tailed settings. Our GSLOPE/TSLOPE
comparison fits into this agenda by quantifying, through limiting distributions, how tail
heaviness affects the score covariance and how loss specification interacts with SLOPE-type
regularization.

%%%%%%%%%%%%%%%%%%%%%%%%%%%%%%%%%%%%%%%%%%%%%%%%%%%%%%%%%%%%%%
\section{Theoretical results}\label{sec:theoretical results}
\subsection{Precision estimation by SLOPE}
Let $X^{(1)}, \dots, X^{(n)}$ be $n$ i.i.d.\ observations of a random vector $X = (X_1, \dots, X_p)^T \in \mathbb{R}^p$. We assume throughout that $\mathbb{E}[X] = 0$ and that the covariance matrix $\operatorname{Cov}(X)$ has full rank. We consider regularized M-estimators targeting an underlying structural parameter $\Theta_0$, defined as solutions to empirical optimization problems of the form
\begin{align}\label{main objective} \widehat{\Theta}_n \in \underset{\Theta \in \mathrm{Sym}+(p)}{\operatorname{argmin}}\left\{ \frac{1}{n} \sum_{i=1}^n \ell\big(X^{(i)},\Theta\big) + n^{-1/2} \operatorname{Pen}(\Theta) \right\},
\end{align}
where $\ell(X,\Theta)$ denotes a prespecified loss function and $\operatorname{Pen}(\Theta)$ is a suitable, problem-specific regularizer. The target parameter $\Theta_0$ is fundamentally defined as the unique minimizer of the population risk, $\mathbb{E}[\ell(X,\Theta)]$, and often encodes conditional independencies (e.g., coinciding with the precision matrix under Gaussianity). Here, $\mathrm{Sym}_+(p)$ denotes the set of $p \times p$ positive definite matrices.

For example, in the Gaussian model $X \sim \mathcal{N}(0,\Sigma)$, the negative log-likelihood loss is given by
\begin{align}\label{gaussian loss}
    \ell(x, \Theta) := -\frac{1}{2} \log \det (\Theta) + \frac{1}{2} x^T \Theta x.
\end{align}
The most widely studied regularization is the graphical Lasso~\cite{friedman2008sparse},
\[
\operatorname{Pen}(\Theta) = \lambda \sum_{i<j} |\Theta_{ij}|, \quad \lambda > 0,
\]
for which extensive theory on estimation and selection consistency has been developed~\cite{RavikumarWainwrightRaskuttiYu2011}. Recent work by~\cite{graczyk2025graphical} extends this study to more general polyhedral and atomic norms, including the Fused Lasso, Group Lasso, and the SLOPE norm. 

Let 
$
\theta_+:=\vec_+(\Theta)=(\Theta_{ij})_{i>j}    
$
denote the vectorization of the strictly lower-triangular components of $\Theta$ as in \cite{graczyk2025graphical}. The SLOPE norm is determined by an ordered sequence of hyperparameters $\lambda_{1}>\lambda_2>\dots>\lambda_{p(p-1)/2}>0$,
\[
\operatorname{Pen}(\Theta) = \sum_{j=1}^{p(p-1)/2} \lambda_j |\vec_+(\Theta)|_{(j)},
\]
where $|\vec_+(\Theta)|_{(1)} \geq |\vec_+(\Theta)|_{(2)} \geq \dots \geq |\vec_+(\Theta)|_{(p(p-1)/2)}$ denote the ordered absolute values of the strictly lower-triangular entries of $\Theta$. 
Unlike the Lasso penalty, graphical SLOPE can cluster coefficients; i.e., it can yield estimates $\widehat{\Theta}_{ij}$ that are equal in magnitude. We are interested in dimensionality reduction by identifying the SLOPE pattern of the vector $\theta_{+}=\vec_+(\Theta)$, given by
\begin{equation*}
    \mathbf{patt}(\theta_+)_i=\operatorname{rank}(|\theta_+|)_i\operatorname{sign}(\theta_+)_i, \quad \forall i\in\{1,\dots,p(p-1)/2\}.
\end{equation*}
As illustrated in Example~\ref{example: pattern and pattern space} in the Appendix, the proper identification of the SLOPE pattern allows for the proper identification of true edges as well as for the identification of the groups of edges with the same absolute value of the elements on the precision matrix. This allows for a more efficient dimensionality reduction and improved estimates of $\Theta_0$ by using appropriate constraints in the corresponding optimization problem. More generally, one can formally define the pattern of a polyhedral penalty $\operatorname{Pen}(x)$ in terms of its subdifferential via the relation 
\[
\mathbf{patt}(x)= \mathbf{patt}(y) \iff \partial \operatorname{Pen}(x)=\partial \operatorname{Pen}(y),
\]
see \cite{bogdan2022pattern, graczyk2025graphical, hejny2025asymptotic, hejny2025unveiling}. For notational simplicity, we abbreviate $\mathbf{patt}(\vec_+(\Theta))$ as $\mathbf{patt}(\Theta)$, by which we refer strictly to the $p(p-1)/2$-dimensional vector corresponding to the pattern of the lower-triangular matrix entries.

In this work, we fix the dimension $p$ of the vector $X=(X_1,\dots,X_p)^T$ and analyze the asymptotic behavior of the rescaled error
\begin{equation*}
    \sqrt{n}(\widehat{\Theta}_n-\Theta_0) \quad \text{ as } \quad n\to\infty.
\end{equation*}
This classical asymptotic regime is studied for the Lasso estimator in linear models in \cite{fu2000asymptotics} and developed in greater generality in \cite{hejny2025asymptotic}. The asymptotic distribution of the error is expressed in terms of the directional derivative at $\Theta_0$ in the direction $U\in\mathrm{Sym}(p)$, defined by
\begin{equation*}
\operatorname{Pen}'(\Theta_0;U)=\lim_{\varepsilon \downarrow 0}\frac{1}{\varepsilon}(\operatorname{Pen}(\Theta_0+\varepsilon U)-\operatorname{Pen}(\Theta_0)).
\end{equation*}
Note that by the convexity of the penalty $\operatorname{Pen}(\Theta)$, the directional derivative always exists. For an explicit expression of the directional derivative for graphical SLOPE, we refer to Appendix~\ref{appendix: directional derivative}. The key asymptotic result is Corollary~3.4 in \cite{hejny2025asymptotic}, which we apply in the context of the graphical Lasso and SLOPE. For completeness, we restate the result in its graphical form as Theorem~\ref{main conditions theorem}. Our focus is on the SLOPE penalty, although the result remains valid for any polyhedral regularizer as defined in \cite{hejny2025asymptotic} or \cite{graczyk2025graphical}.

Let $\widehat{\Theta}_n$ denote the minimizer of the objective function in \eqref{main objective}, characterized by a loss function $\ell(x,\Theta)$ and the SLOPE penalty $\operatorname{Pen}(\Theta)$. Let $\Theta_0$ minimize the risk function
\begin{equation*}
    G(\Theta)=\mathbb{E}[\ell(X,\Theta)].
\end{equation*}

%Throughout this work, we define our target parameter $\Theta_0$ conceptually as the unique minimizer of the population risk $\mathbb{E}[\ell(X,\Theta)]$ for a chosen loss function. In the context of elliptical distributions, it is natural to express the data in the canonical form $X = \Theta_0^{-1/2}uR$, where $\Theta_0$ acts as a structural shape parameter. Crucially, because $\Theta_0$ is uniquely defined by the loss function, this representation imposes a specific compatibility condition on the radial component $R$. We will see that aligning the risk minimizer with this shape parameter requires $\mathbb{E}[R^2]=p$ under the Gaussian loss, and $\mathbb{E}[R^2/(\nu+R^2)]=p/(\nu+p)$ under the $t$-loss. When these conditions are violated, data rescaling is required to restore consistency.

Assume there exists an open neighborhood $B$ of $\Theta_0$ such that the map $\Theta\mapsto\ell(x,\Theta)$ is twice continuously differentiable on $B$ for every $x$. We further impose the following regularity conditions:

\begin{enumerate}[label=(\roman*)]
    \item For all $\Theta \in B$, $\lVert\nabla^2_\Theta \ell(X, \Theta)\rVert\leq M(X)$ for some $M$ with $\mathbb{E}[M(X)^2]<\infty$.
    \item $G(\Theta)$ is thrice continuously differentiable on $B$, and the Hessian $C:=\nabla^2 G(\Theta_0)$ is positive definite.
    \item The expected gradient satisfies $\mathbb{E}[\nabla_\Theta \ell(X, \Theta_0)]=0$, and the covariance matrix
    \begin{equation*}
        C_{\triangle} :=\mathbb{E}\left[\vec(\nabla_{\Theta}\ell(X,\Theta_0)) \vec(\nabla_{\Theta}\ell(X,\Theta_0))^T\right]
    \end{equation*}
    is finite.
    \item $\Theta_0$ lies in the interior of $\operatorname{Sym}_+(p)$, and the sequence $(\widehat{\Theta}_n)_{n\in\mathbb{N}}$ is uniformly tight.
    \item For any compact $K\subset \operatorname{Sym}_+(p)$, $\sup_{\Theta\in K}|\ell(X,\Theta)|\hspace{-0.1cm}\leq \hspace{-0.1cm} L(X)$ for some $L$ with $\mathbb{E}[L(X)]\hspace{-0.1cm}<\hspace{-0.1cm}\infty$.
\end{enumerate}

\begin{theorem}\label{main conditions theorem}
    Assume that conditions (i)--(v) hold. Then, the sequence $\widehat{U}_n=\sqrt{n}(\widehat{\Theta}_n-\Theta_0)$ converges in distribution to
\begin{equation}\label{main asymptotic distribution}
 \widehat{U}=\underset{U\in\operatorname{Sym}(p)}{\operatorname{argmin}}\left\{ \frac{1}{2}\vec(U)^T C\, \vec(U) - W^T \vec(U) + \operatorname{Pen}'(\Theta_0; U)\right\},
\end{equation}
where $C=\nabla^2 G(\Theta_0)$, $W\sim\mathcal{N}_{p^2}(0,C_{\triangle})$, and $C_{\triangle}$ is the covariance matrix defined in (iii). 
Moreover, 
\begin{equation}\label{pattern convergence}
    \lim_{n\to\infty}\mathbb{P}\left[\mathbf{patt}\big(\sqrt{n}(\widehat{\Theta}_n-\Theta_0)\big)=\mathfrak{p}\right]=\mathbb{P}\left[\mathbf{patt}(\widehat{U})=\mathfrak{p}\right],
\end{equation}
for every SLOPE pattern $\mathfrak{p}\in \{\mathbf{patt}(U): U\in \operatorname{Sym}(p)\}$.
\end{theorem}
The theorem follows from Corollary 3.4~\cite{hejny2025asymptotic}.

\subsection{Graphical SLOPE with the Gaussian loss (GSLOPE)}\label{sec:graphical SLOPE with Gaussian loss}

The following result is a direct consequence of Theorem~\ref{main conditions theorem} and applies beyond normally distributed data.

\begin{theorem}\label{theorem convergence in distribution}
Let $\widehat{\Theta}_n$ be the graphical SLOPE estimator, which minimizes the objective~\eqref{main objective} with the loss function~\eqref{gaussian loss}. Assume that $\mathbb{E}[X_i^4] < \infty$ for all $i\in\{1,\dots,p\}$. Then, for the target parameter $\Theta_0 = (\mathbb{E}[XX^T])^{-1}$, the rescaled error $\sqrt{n}(\widehat{\Theta}_n-\Theta_0)$ converges in distribution to the minimizer $\widehat{U}$ defined in \eqref{main asymptotic distribution},
%\begin{align}\label{main asymptotic distribution}
%\hat{U} = \underset{U\in\mathrm{Sym}(p)}{\operatorname{argmin}}\left\{ \frac{1}{2}\vec(U)^T C \vec(U) - W^T \vec(U) + \operatorname{Pen}'(\Theta_0; U)\right\},
%\end{align}
where
\begin{align*}
    C&=\frac12(\Sigma\otimes\Sigma),\\
    C_{\triangle}&= \Cov(\vec(XX^T))/4.
\end{align*} %and $W\sim\mathcal{N}_{p^2}(0,C_{\triangle})$ with $C_{\triangle}= \frac{1}{4}\Cov(\vec(XX^T))$. Moreover,
Moreover, the SLOPE pattern of the error converges in the sense \eqref{pattern convergence}.
%\begin{equation*}
%\lim_{n\to\infty}\mathbb{P}\left[\mathbf{patt}\left(\sqrt{n}(\widehat{\Theta}_n-\Theta_0)\right)=\mathfrak{p}\right]=\mathbb{P}\left[\mathbf{patt}(\widehat{U})=\mathfrak{p}\right],
%\end{equation*}
%for every SLOPE pattern $\mathfrak{p}\in \{\mathbf{patt}(U): U\in \operatorname{Sym}(p)\}$.
\end{theorem}

In the special case where $X \sim \mathcal{N}_p(0,\Sigma)$, the theorem above holds with $C_{\triangle} = C$. Furthermore, since $\Sigma$ is positive definite, $C$ is positive definite. Consequently, the objective in \eqref{main asymptotic distribution} is strictly convex, as it consists of a strictly convex quadratic form and the convex directional derivative $U\mapsto\operatorname{Pen}'(\Theta_0; U)$. It follows that the minimization problem \eqref{main asymptotic distribution} admits a unique solution.

As a primary application of Theorem~\ref{theorem convergence in distribution}, we consider the case where $X$ follows an elliptical distribution
\begin{equation*}
    X=\Sigma^{1/2}uR,
\end{equation*}
where $u$ is uniformly distributed on the unit sphere $\mathbb{S}^{p-1} \subset \mathbb{R}^p$ and is independent of the random radius $R\geq 0$. Observe that since $\mathbb{E}[u]=0$, we have $\mathbb{E}[X]=0$. Furthermore, using the fact that $\mathbb{E}[uu^T]=p^{-1}I_p$, we obtain
\begin{equation*}
    \Cov(X)=\mathbb{E}[R^2\Sigma^{1/2}uu^T\Sigma^{1/2}]=\mathbb{E}[R^2]p^{-1}\Sigma=\Sigma,
\end{equation*}
provided that the radial component satisfies $\mathbb{E}[R^2]=p$. We obtain the following asymptotic result for elliptical distributions:

\begin{proposition}\label{proposition for elliptical distributions}
    Assume that $X=\Sigma^{1/2}uR$ follows an elliptical distribution with $\mathbb{E}[R^2]=p$. Then the asymptotic distribution is determined by \eqref{main asymptotic distribution}, with $C=(\Sigma\otimes\Sigma)/2$ and $C_{\triangle}$ given by:
    \begin{align}\label{covariance for elliptical distributions}
        C_{\triangle} = C + \left(\frac{\mathbb{E}[R^4]}{p(p+2)}-1\right) \left(C + \frac{1}{4} \vec(\Sigma)\vec(\Sigma)^T\right).
    \end{align}
\end{proposition}
We establish the identity~\eqref{covariance for elliptical distributions} in Appendix~\ref{appendix: proofs for Gaussian Loss}.

\begin{example} (Gaussian distribution) The most common setting is the Gaussian model $X \sim \mathcal{N}_p(0,\Sigma)$. A key property of this model is that the sparsity pattern of the precision matrix $\Theta_0 := \Sigma^{-1}$ encodes conditional independencies: $(\Theta_0)_{ij} = 0$ if and only if $X_i \indep X_j \mid X_{-(i,j)}$ (see, for example, \cite{lauritzen1996graphical}). In the Gaussian case $X \sim \mathcal{N}_p(0,\Sigma)$, the squared radius follows a chi-squared distribution, $R^2 \sim \chi^2_p$, with $\mathbb{E}[R^2] = p$. Moreover, since $\mathbb{E}[R^4]=p(p+2)$, the correction term in \eqref{covariance for elliptical distributions} vanishes, and the asymptotic distribution \eqref{main asymptotic distribution} holds with 
    \begin{equation*}
        C_{\triangle}=\frac{1}{2}(\Sigma\otimes\Sigma)=C.
    \end{equation*}
    Alternatively, the fact that $C=C_{\triangle}$ follows from the standard result that, under correct model specification, the covariance of the score equals the expected Hessian of the loss function.
\end{example}

When the condition $\mathbb{E}[R^2]=p$ is not satisfied, such as for the multivariate t-distribution, where $\mathbb{E}[R^2]/p = \nu/(\nu-2)$, the theoretical machinery still applies, but necessitates either a rescaling of the data or an adjustment to the target parameter. We illustrate these dynamics in the following example.

\begin{example}\label{t-distributed data gaussian loss}
(Multivariate $t$-distribution) Building on the previous proposition, we now relax the assumption that $\mathbb{E}[R^2]=p$ to consider a general elliptical distribution $X=\Theta_0^{-1/2}uR$. However, for the minimizer of the Gaussian population risk to correctly identify the true parameter $\Theta_0$ governing the data distribution, we require $\mathbb{E}[XX^T]=\Theta_0^{-1}$, which holds if and only if $\mathbb{E}[R^2]=p$. If this condition is violated, the Gaussian loss targets a scaled version of the truth, and by Theorem~\ref{theorem convergence in distribution}, one estimates $\mathbb{E}[XX^T]^{-1}\neq\Theta_0$. Consequently, one must either rescale the target parameter or rescale the data to adjust for the inflation factor $\gamma:=\mathbb{E}[R^2]/p$. We opt to rescale the data as
\begin{equation*}
    \tilde{X}=\gamma^{-1/2}X= \Theta_0^{-1/2}u\tilde{R},
\end{equation*}
where $\tilde{R}=\gamma^{-1/2}R$. Since $\mathbb{E}[\tilde{R}^2]=p$, this setting reduces to the setup in Proposition~\ref{proposition for elliptical distributions} with $X$ and $R$ replaced by $\tilde{X}$ and $\tilde{R}$, respectively. We have $C=(\Sigma\otimes\Sigma)/2$, with $\Sigma = \Theta^{-1}_0= \Cov(\tilde{X})$, and
\begin{align*}
    C_{\triangle}
    & = C + \left(\frac{\gamma^{-2}\mathbb{E}[R^4]}{p(p+2)}-1\right) \left(C + \frac{1}{4} \vec(\Sigma)\vec(\Sigma)^T\right).
\end{align*}

For the multivariate $t$-distribution with degrees of freedom $\nu>4$, represented as $X \sim \mathcal{N}_p(0,\Sigma)/\sqrt{\chi^2_\nu/\nu} $, we have $X=\Sigma^{1/2}uR$ with $R=\sqrt{\chi^2_p/(\chi^2_{\nu}/\nu)}$. The scaling constant is $\gamma=\mathbb{E}[R^2]/p=\nu / (\nu-2)$. Furthermore,
\begin{equation*}
    \mathbb{E}[R^4]=\nu^2\mathbb{E}[(\chi^2_p)^2]\mathbb{E}[(\chi^2_{\nu})^{-2}] = \nu^2\frac{p(p+2)}{(\nu-2)(\nu-4)}<\infty,
\end{equation*}
where we used the identity $\mathbb{E}[(\chi^2_{\nu})^{-2}]=((\nu-2)(\nu-4))^{-1}$. This yields the following expression for the rescaled data $\tilde{X}=\gamma^{-1/2}X=\sqrt{1-2/\nu} X$:
\begin{align*}
    C_{\triangle}
    & = C + \frac{2}{\nu-4} \left(C + \frac{1}{4} \vec(\Sigma)\vec(\Sigma)^T\right).
\end{align*}
We remark that $C_{\triangle}-C$ is positive definite; hence, the asymptotic covariance for the $t$-distributed data is larger (in the partial order on symmetric matrices) than that of the correctly specified Gaussian model. However, this difference is of order $O(1/\nu)$ as $\nu\to\infty$.
\end{example}

 \subsection{Graphical SLOPE with loss induced by the t-distribution (TSLOPE)}\label{sec:graphical SLOPE with t-loss}

Consider the multivariate $t$-distribution
\begin{equation*}
    X \sim \frac{\mathcal{N}_p(0,\Sigma)}{\sqrt{\chi^2_\nu / \nu}},
\end{equation*}
with $\nu > 2$, where $\mathcal{N}_p(0,\Sigma)$ and $\chi^2_{\nu}$ denote independent random variables. This construction represents a Gaussian vector scaled by an independent scalar derived from the inverse chi-squared distribution. Equivalently, if $Z \sim \mathcal{N}_p(0, I_p)$ and $v = (\chi^2_\nu / \nu)^{-1/2}$ are independent, then $X = \Theta_0^{-1/2} Z\, v$ follows this distribution, with covariance $\operatorname{Cov}(X)=\frac{\nu}{\nu - 2}\Sigma$.

Normalizing $Z$ to obtain a unit vector $u = Z / \lVert Z \rVert$ and separating the radial component yields the canonical elliptical representation $X = \Theta_0^{-1/2} u\, R_0$. Here, $R_0 = \lVert Z \rVert v \sim \sqrt{\chi^2_p/(\chi^2_{\nu}/\nu)}$ is independent of $u$. 

More generally, we consider elliptical distributions of the form 
\begin{equation*}
    X = \Theta_0^{-1/2} u\, R,
\end{equation*}
where $u$ is uniformly distributed on the unit sphere $\mathbb{S}^{p-1}$ and $R$ is a nonnegative random variable representing the radius. For the $t$-distribution, the negative log-likelihood induces, up to an additive constant, the loss function
\begin{equation}\label{t-distribution loss}
    \ell(x, \Theta) = -\frac{1}{2} \log\det(\Theta) 
    + \frac{\nu + p}{2} \log\big(\nu + x^\top \Theta x\big).
\end{equation}
We consider the minimizer $\widehat{\Theta}_n$ of \eqref{main objective} with the loss defined in \eqref{t-distribution loss}. The subsequent theorem establishes that $\widehat{\Theta}_n$ consistently estimates the parameter $\Theta_0$ of the elliptical distribution $X=\Theta_0^{-1/2}uR$, provided that
\begin{equation}\label{key assumption on R for t-distribution}
    \mathbb{E}\left[\frac{R^2}{\nu+R^2}\right]=\frac{p}{\nu+p}.
\end{equation}
Note that the parameter $\Theta_0$ is not the precision matrix $\operatorname{Cov}(X)^{-1}$ of the $t$-distributed vector $X$, but rather a scalar multiple thereof:
\begin{equation*}
    \operatorname{Cov}(X)^{-1} = \frac{p}{\mathbb{E}[R^2]}\Theta_0.
\end{equation*}
Consequently, the estimated precision matrix can be recovered as a scalar multiple of $\widehat{\Theta}_n$. For the special case $R=R_0$, where $\mathbb{E}[R^2]/p=\nu/(\nu-2)$, the precision is given by $\operatorname{Cov}(X)^{-1} = (1-2/\nu)\Theta_0$.
\begin{lemma}\label{lemma t-hessian,covariance}
    Assume that $X=\Theta_0^{-1/2}uR$ follows an elliptical distribution satisfying \eqref{key assumption on R for t-distribution}. %which satisfies $\mathbb{E}[(\nu+p)R^2/(p(\nu+R^2))]=1$.
    Then the Hessian $C:=  \nabla_{\Theta}^2\mathbb{E}[\ell(X, \Theta_0)]$, and the covariance $C_{\triangle}:=\operatorname{Cov}(\vec(\nabla_{\Theta}\ell(X, \Theta_0)))$ have explicit representations as
    \begin{align}
    C &= \frac{1}{2}(\Sigma\otimes\Sigma) - \frac{\mathbb{E}[\xi_R]}{p(p+2)}\left(\Sigma\otimes\Sigma+\frac{1}{2}\vec(\Sigma)\vec(\Sigma)^T \right),\label{t-hessian}\\
    C_{\triangle} &=  C + \left(\frac{\nu+p+2}{p(p+2)}\mathbb{E}[\xi_R]-1\right)\left( \frac{1}{2}(\Sigma\otimes\Sigma)+\frac{1}{4}\vec(\Sigma)\vec(\Sigma)^T\right),\label{t-covariance}
\end{align}
where $\Sigma=\Theta_0^{-1}$, and $\xi_{R}=(\nu+p)R^4/(\nu+R^2)^2$. Moreover, the Hessian $C$ is positive definite.
\end{lemma}

\begin{remark}\label{t-distributed data t-loss}
Observe that for the special case of a $t$-distributed variable $X\sim \mathcal{N}_p(0,\Sigma)/\sqrt{\chi^2_\nu/\nu}$, where $R=R_0\sim\sqrt{\chi^2_p/(\chi^2_\nu/\nu)}$ and $\mathbb{E}\left[\xi_{R_0}\right]=p(p+2)/(\nu+p+2)$, the quantities simplify to
\begin{align*}
    C = C_{\triangle}= \frac{1}{2}(\Sigma\otimes\Sigma) - \frac{1}{\nu+p+2}\left(\Sigma\otimes\Sigma+\frac{1}{2}\vec(\Sigma)\vec(\Sigma)^T \right).
\end{align*}
We observe that as $\nu \to \infty$, the limiting quantities approach $(\Sigma\otimes\Sigma)/2$, as in the Gaussian model.
\end{remark}

\begin{theorem}\label{theorem convergence in distribution for t distribution}
Assume that $X=\Theta_0^{-1/2}uR$ follows an elliptical distribution satisfying \eqref{key assumption on R for t-distribution}, and $\mathbb{E}[\log(1+R^2)]<\infty$. %which satisfies $\mathbb{E}[(\nu+p)R^2/(p(\nu+R^2))]=1$.
The rescaled error $\sqrt{n}(\widehat{\Theta}_n-\Theta_0)$ converges in distribution to
\begin{align}\label{main asymptotic distribution t distribution}
\hat{U}=\underset{U\in\mathrm{Sym}(p)}{\operatorname{argmin}}\left\{ \frac{1}{2}\vec(U)^T C\, \vec(U) - W^T \vec(U) + \operatorname{Pen}'(\Theta_0; U)\right\},
\end{align}
where $W\sim\mathcal{N}_{p^2}(0,C_{\triangle})$ with $C$ and $C_{\triangle}$ given by \eqref{t-hessian} and \eqref{t-covariance}.

Moreover, 
\begin{equation*}
\lim_{n\to\infty}\mathbb{P}\left[\mathbf{patt}(\sqrt{n}(\widehat{\Theta}_n-\Theta_0))=\mathfrak{p}\right]=\mathbb{P}\left[\mathbf{patt}(\widehat{U})=\mathfrak{p}\right],
\end{equation*}
for every SLOPE pattern $\mathfrak{p}\in \{\mathbf{patt}(U): U\in \operatorname{Sym}(p)\}$.

\end{theorem}

\section{Simulations}\label{sec:simulations}

To estimate the lower-triangular coefficients of the precision $\Theta_0$, we use the SLOPE estimator with the penalty function:
\[
\mathrm{Pen}(\Theta) =\alpha \cdot \sum_{j=1}^{p(p-1)/2} \lambda_j |\Theta|_{(j)},
\]
where $\alpha>0$ is the penalty strength, $|\Theta|_{(1)} \geq |\Theta|_{(2)} \geq \dots \geq |\Theta|_{(p(p-1)/2)}$ denote the ordered absolute values of the strictly lower-triangular entries of $\Theta$. We use the BH($q$)-normalized penalty sequence:
\begin{equation}\label{BH-sequence normalized}
\lambda_j = \Phi^{-1} \left(1 - \frac{qj}{p(p-1)} \right)/\bar{\lambda}, \quad j = 1, 2, \dots, p(p-1)/2,
\end{equation}
where $\Phi^{-1}(\cdot)$ denotes the quantile function of the standard normal distribution, $q\in (0,1)$, and the normalization constant is given by
\begin{equation*}
    \bar{\lambda}=\frac{2}{p(p-1)}\sum_{i=1}^{p(p-1)/2}\Phi^{-1} \left(1 - \frac{qj}{p(p-1)} \right).
\end{equation*}
The constant $\bar{\lambda}$ rescales the penalty sequence such that the average penalty coefficient equals $1$. Figure~\ref{fig:RMSE} illustrates the Root Mean Square Error (RMSE) of the estimator:
\begin{equation}\label{finite n RMSE}
    \operatorname{RMSE}(\widehat{\Theta}_n) = \sqrt{\mathbb{E}\left[ \lVert\vech(\widehat{\Theta}_n) - \vech(\Theta_0)\rVert_2^2\right]},
\end{equation}
where $\widehat{\Theta}_n$ is the SLOPE estimator of $\Theta_0$ and $\vech(\cdot)$ denotes the vectorization of the lower-triangular part (including the diagonal). We set $q=0.5$ and plot the $\operatorname{RMSE}(\widehat{\Theta}_n)$ as a function of the tuning parameter $\alpha$ for various sample sizes $n$. We further compare the empirical RMSE of $\widehat{\Theta}_n$ to its asymptotic counterpart
\begin{equation}\label{asymptotic RMSE}
   \lim_{n\to\infty}\sqrt{n}\operatorname{RMSE}(\widehat{\Theta}_n)=\sqrt{\mathbb{E}\left[\lVert\vech(\widehat{U})\rVert^2_2\right]}.
\end{equation}
The term on the right-hand side is approximated by averaging the squared norms of $M=100$ independent realizations of the vector $\widehat{U}$, which are obtained by solving the asymptotic optimization problem in equation~\eqref{main asymptotic distribution} for corresponding realizations of the noise vector $W$.

We assume the data follows a multivariate Gaussian distribution $X \sim \mathcal{N}_p(0,\Sigma)$, where the covariance matrix has a block-diagonal structure $\Sigma= I_2 \otimes \Sigma^0$. Here, $\Sigma^0$ is a $10\times10$ matrix with compound symmetry, where $\Sigma^0_{i,i} = 1$ and $\Sigma^0_{i,j} = 0.2$ for $i \neq j$. $I_2$ denotes the $2\times2$ identity matrix, and $\otimes$ denotes the Kronecker product.

\begin{figure}[ht]
    \centering
    {\includegraphics[width=0.8\textwidth]{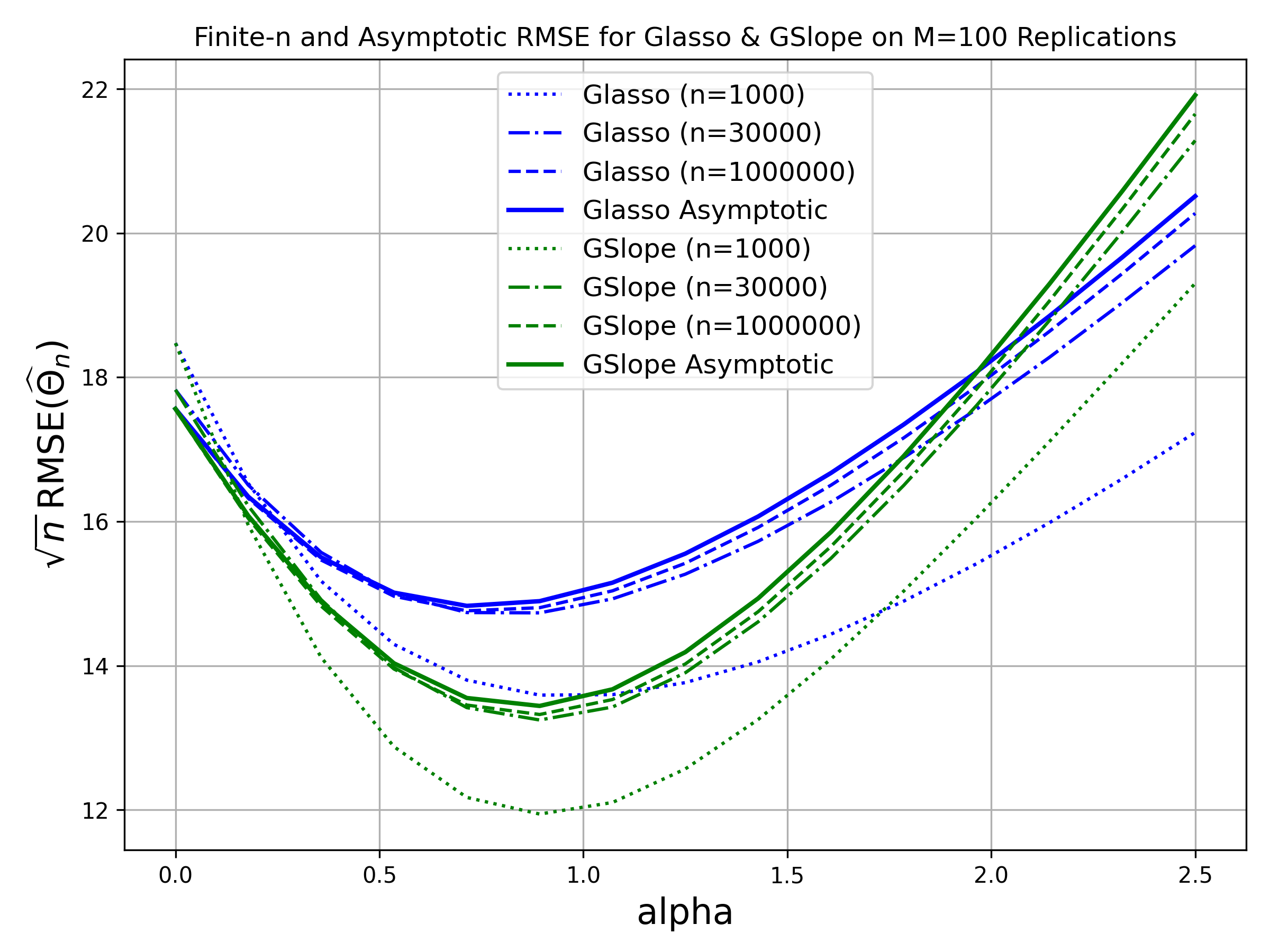}}
   \caption{Convergence of the rescaled empirical $\operatorname{RMSE}$, defined as $\sqrt{n} \cdot \operatorname{RMSE}(\widehat{\Theta}_n)$ (where $\operatorname{RMSE}$ is given by Equation \eqref{finite n RMSE}), to the asymptotic error (Equation \eqref{asymptotic RMSE}) as a function of the sample size $n$. The figure compares the performance of Graphical Lasso and Graphical SLOPE using the BH($q=0.5$) normalized sequence.}
    \label{fig:RMSE}
\end{figure}

\subsection{Clustering Measure}
One method to quantify how well the estimator $\widehat{\Theta}_n$ recovers the clusters of $\Theta_0$ is to measure the $\ell_2$-distance from the pattern space of $\Theta_0$. To formalize this, let $\theta^0_{+} = \vec_{+}(\Theta_0) = (\Theta_{ij})_{i > j}$ denote the strictly sub-diagonal vectorization of $\Theta_0$ in $\mathbb{R}^{p(p-1)/2}$. We define the pattern space at $\theta^0_+$ as
\begin{equation*}
    V_0:=\operatorname{span}\{u\in \mathbb{R}^{p(p-1)/2}: \mathbf{patt}(u)=\mathbf{patt}(\theta^0_{+})\}.
\end{equation*}
For a more comprehensive analysis of the pattern space, we refer to \cite{hejny2025asymptotic}. 

Let $P$ denote the orthogonal projection onto $V_0$, and consider the strictly sub-diagonal error $\hat{u}_n:=\sqrt{n}(\vec_+(\widehat{\Theta}_n)-\vec_+(\Theta_0))$. This error decomposes as 
\begin{equation*}
    \lVert \hat{u}_n\rVert_2^2=\underbrace{\lVert P\hat{u}_n\rVert_2^2}_{\text{bias error}}+\underbrace{\lVert (I-P)\hat{u}_n\rVert_2^2}_{\text{clustering error}}.
\end{equation*}
The clustering error corresponds to the squared $\ell_2$-distance of $\hat{u}_n$ from the subspace $V_0$. For an illustration of this concept, we refer to Example~\ref{example: clustering error} in the appendix. We define the root-$n$ rescaled clustering RMSE as:
\begin{equation*}
    \operatorname{RMSE-CL}(\widehat{\Theta}_n):=\sqrt{\mathbb{E}\left[\lVert (I-P)\hat{u}_n\rVert_2^2\right]}.
\end{equation*}
Analogously, using the limiting variable $\hat{u}:=\vec_+(\widehat{U})$ derived from \eqref{main asymptotic distribution}, we define the asymptotic clustering error as
\begin{equation}\label{clustering error asymptotic}
    \operatorname{RMSE-CL}:=\sqrt{\mathbb{E}\left[\lVert (I-P)\hat{u}\rVert_2^2\right]}.
\end{equation}
This quantity represents the asymptotic limit of $\operatorname{RMSE-CL}(\widehat{\Theta}_n)$.

In Figure~\ref{fig:RMSE_cluster}, we compare the asymptotic strictly sub-diagonal Root Mean Squared Error, given by $(\mathbb{E}[\lVert\vec_{+}(\widehat{U})\rVert_2^2])^{1/2}$, with the clustering error $(\mathbb{E}[\lVert(I-P)\vec_{+}(\widehat{U})\rVert_2^2])^{1/2}$ defined in \eqref{clustering error asymptotic}, as a function of the tuning parameter $\alpha$. The comparison includes the graphical Lasso and graphical SLOPE; for the latter, we use the normalized BH sequence \eqref{BH-sequence normalized} with $q=0.5$.

\begin{figure}[ht]
    \centering
    {\includegraphics[width=0.6\textwidth]{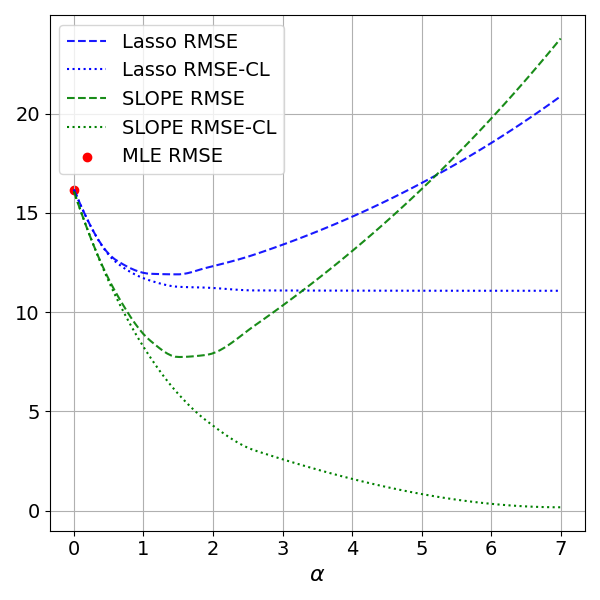}}
    \caption{Comparison of the asymptotic clustering error $\operatorname{RMSE-CL}$ (Equation \eqref{clustering error asymptotic}) versus the asymptotic total $\operatorname{RMSE}$ (Equation \eqref{asymptotic RMSE}) for Graphical Lasso and Graphical SLOPE with the normalized BH($q=0.5$) sequence.}
    \label{fig:RMSE_cluster}
\end{figure}

\begin{figure}[ht]
    \centering
    
    % --- First Subfigure ---
    \begin{subfigure}[b]{0.32\textwidth}
        \centering
        % Replace with your filename for nu=3
        \includegraphics[width=\linewidth]{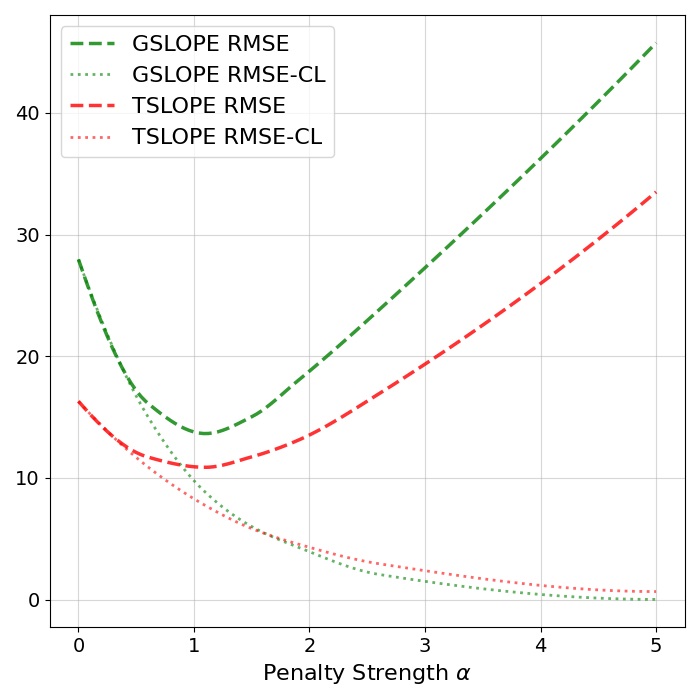}
        \caption{$\nu = 5$}
        \label{fig:nu3}
    \end{subfigure}
    \hfill % Adds flexible space between images
    % --- Second Subfigure ---
    \begin{subfigure}[b]{0.32\textwidth}
        \centering
        \includegraphics[width=\linewidth]{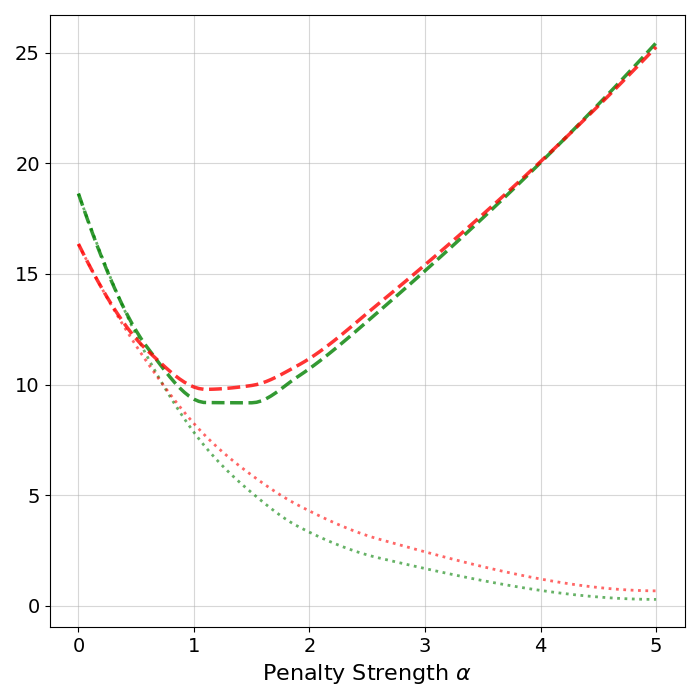}
        \caption{$\nu = 10$}
        \label{fig:nu5}
    \end{subfigure}
    \hfill % Adds flexible space between images
    % --- Third Subfigure ---
    \begin{subfigure}[b]{0.32\textwidth}
        \centering
        % Replace with your filename for nu=10
        \includegraphics[width=\linewidth]{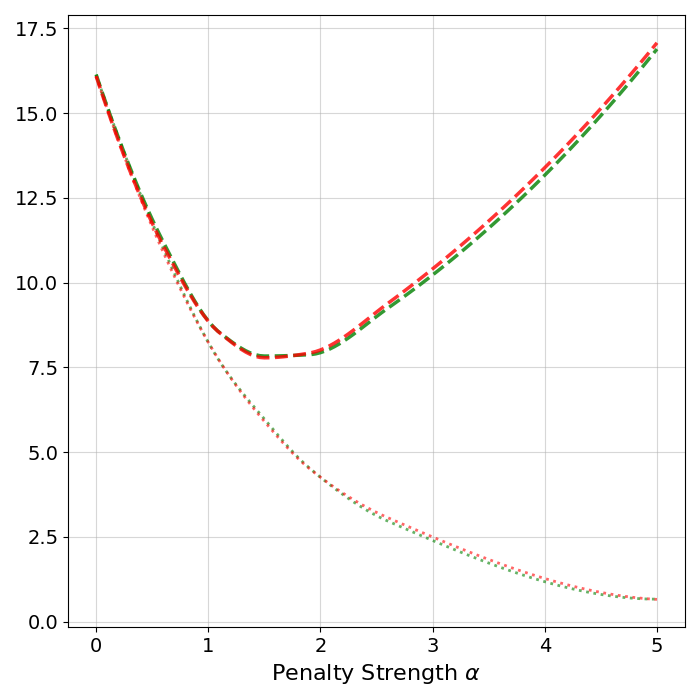}
        \caption{$\nu = 1000$}
        \label{fig:nu10}
    \end{subfigure}

    \caption{ Comparison of the asymptotic GSLOPE and TSLOPE errors, defined as $\lim_{n\to\infty}\mathbb{E}[\lVert\vech(\sqrt{n}(\widehat{\Theta}_n-\Theta_0)) \rVert_2^2]$, for heavy-tailed $t$-distributed data. The degrees of freedom are $\nu=5$ (left), $\nu=10$ (center), and $\nu=1000$ (right). TSLOPE uses the normalized BH($q=0.5$) sequence. The GSLOPE penalty is rescaled by factors of $2.7$, $1.5$, and $1$, respectively, such that the minima of the curves match.}
    \label{fig:RMSE_GLOPE_TSLOPE}
\end{figure}

Figure~\ref{fig:RMSE_GLOPE_TSLOPE} compares the asymptotic error of GSLOPE (using the Gaussian loss \eqref{gaussian loss}) and TSLOPE (using the $t$-distribution loss \eqref{t-distribution loss}) when the data is generated by a multivariate $t$-distribution with degrees of freedom $\nu \in \{5, 10, 1000\}$ and covariance parameter $\Sigma=\Theta^{-1}_0$. The asymptotic distribution for both methods is given by \eqref{main asymptotic distribution}, which is fully characterized by the Hessian $C$, the covariance $C_{\triangle}$, and the penalty sequence. We set the penalty sequence to the normalized BH sequence \eqref{BH-sequence normalized} with $q=0.5$.

To ensure GSLOPE consistently estimates $\Theta_0$, we rescale the data by $\sqrt{1-2/\nu}$ prior to estimation. The asymptotic RMSE for GSLOPE is determined by the matrices:
\begin{align*}
    C=\frac{1}{2}(\Sigma\otimes\Sigma)\quad \text{and} \quad C_{\triangle}  
    & = C + \frac{2}{\nu-4} \left(C + \frac{1}{4} \vec(\Sigma)\vec(\Sigma)^T\right),
\end{align*}
as detailed in Example~\ref{t-distributed data gaussian loss}. For the correctly specified minimization using the $t$-loss, the RMSE is determined by
\begin{align*}
    C = C_{\triangle}= \frac{1}{2}(\Sigma\otimes\Sigma) - \frac{1}{\nu+p+2}\left(\Sigma\otimes\Sigma+\frac{1}{2}\vec(\Sigma)\vec(\Sigma)^T \right),
\end{align*}
following Remark~\ref{t-distributed data t-loss}. As expected, for $\nu=5$, the inflation factor $2/(\nu-4)=2$ in GSLOPE dominates, resulting in a large error. However, this factor decays rapidly; for $\nu=10$, it drops to $1/3$, and, somewhat surprisingly, the correction factor $(\nu+p+2)^{-1}$ in TSLOPE causes a slight reversal in the relative performance in favor of GSLOPE. Finally, for $\nu=1000$, both expressions converge to the Gaussian limit, yielding indistinguishable performance.

\newpage
\subsection{Hidden Factor Structure}\label{sec: hidden factor structure}
In what follows, we want to investigate how GSlope and TSlope, in comparison to GLasso and TLasso, are able to extract the dependence structure among variables that follow a block correlation structure, i.e.  variables having a high dependence among those belonging to the same group and a low dependence to those of other groups. To create such a simulated environment, we draw on the financial literature and utilize a so called hidden factor structure, in which time series of asset prices (for example stock equities) are said to have exposures to various underlying combinations of so called risk factors - i.e. systematic sources of risk for which investors may demand a compensation if they cannot diversify them (see, e.g. \cite{ROSS1976, FAMA1993, FAMA2015}).\\
Our simulation set-up follows closely the set-up in \cite{Kremer2020}: That is, assume that the underlying data generating process follows a hidden factor structure, in which the observed time series is dependent on a linear combination of a set of risk factors, each of which can have a distinct loading. Furthermore, assume that there are $k$ assets, $r$ (risk) factors and $t$ observations with $\bold{F}_{t \times r} = [ \bold{f}_1\ \bold{f}_2\ ...\ \bold{f}_{r}]$ where $\bold{f_{j}}$ is the $t \times 1$ return vector of the $jth$ (risk) factor. Given that $\bold{B}_{r \times k}$ holds the loadings of the individual risk factors, the $t\times k$ hidden factor time series is given as:
\begin{equation}\label{HF_Time_Series}
    \bold{R}_{HF} = \bold{F} \times \bold{B} + \bold{\epsilon}
\end{equation}
where $\bold{\epsilon}$ is a $t \times k$ matrix of error terms.\\
For our simulation, we assume the following set-up:
\begin{enumerate}
    \item let $t=1000$, $k=300$, $r=6$, and let there be three distinct groups of assets
    \item each risk factor $j$ is independent from each other and follows a multivariate t distribution $\bold{f}_{j} \sim t_{p}(0, \bold{I}_{r \times r}, \nu)$ with $\nu = 4$ degrees of freedom and where $\bold{I}_{r \times r}$ is the identity matrix
    \item the loading matrix $B_{r \times k}$ is constructed by repeating each of the following column vectors $\frac{k}{3}$ times: $[0.5\ 0.5\ 0\ 0\ 0\ 0]'$, $[0\ 0\ 0.5\ 0.5\ 0\ 0]'$, $[0\ 0\ 0\ 0\ 0.5\ 0.5]'$   
    \item the error vectors $\epsilon_j$, for $j = 1, \ldots, k$, are mutually independent across assets and independent of the risk factors. They are assumed to follow a multivariate $t$-distribution with mean zero, scale matrix $0.15\times I_{k \times k}$, and $\nu = 4$ degrees of freedom.
\end{enumerate}
Given the set-up above, we generate three distinct groups of assets that each have an exposure to exactly two distinct risk factors to which assets from other groups do not have an exposure. Furthermore, from (\ref{HF_Time_Series}) it follows that the covariance matrix of assets can be written as:\\
\begin{equation}
    \bold{\Sigma}_{HF} = \bold{B}'\bold{B} + 0.15 \times \bold{I}_{k\times k}
\end{equation}
Figure \ref{fig:Sim_Oracle_Correlation} plots on the left the so-called true or \textit{oracle} correlation matrix for the chosen hidden factor set-up and on the right the oracle precision matrix, i.e. the inverse of the covariance matrix. The correlation matrix clearly resembles the resulting block correlation structure stemming from the underlying hidden factor exposures, with assets in the same group having a high correlation and zero among assets which do not belong to the same group. Furthermore, also for the oracle precision matrix the hidden factor structure is preserved. As all of the studied methods directly estimate the precision matrix, we investigate how well the respective procedures are able to reproduce the structure of the oracle precision matrix.\\

\begin{figure}[ht]
    \centering
    
    % --- First Subfigure ---
    \begin{subfigure}[b]{0.45\textwidth}
        \centering
        % Replace with your filename for nu=3
        \includegraphics[width=\linewidth]{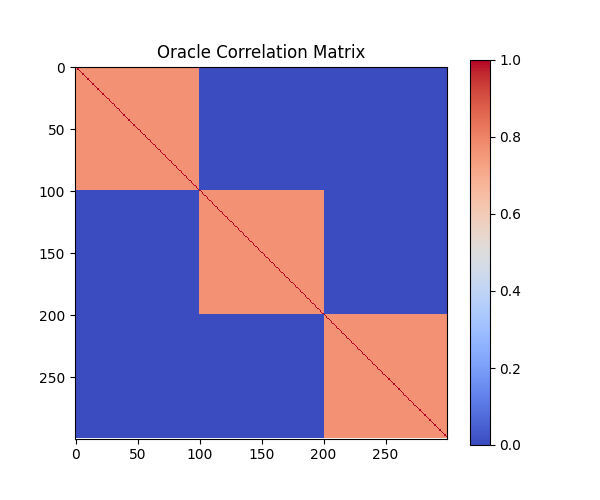}
        %\caption{$\nu = 5$}
        %\label{fig:sim_oracle_precision}
    \end{subfigure}
    \hfill % Adds flexible space between images
    % --- Second Subfigure ---
    \begin{subfigure}[b]{0.45\textwidth}
        \centering
        \includegraphics[width=\linewidth]{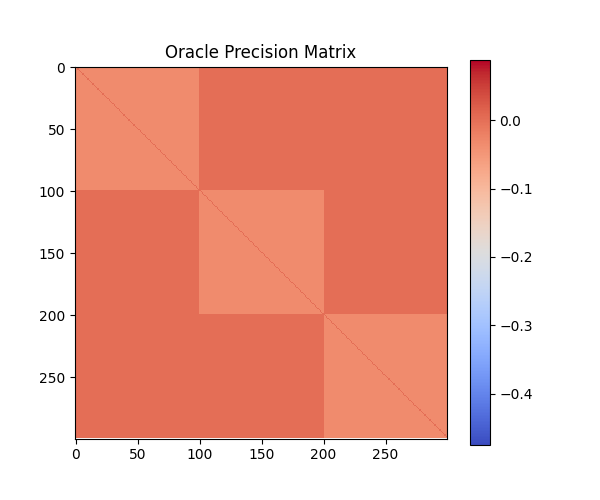}
        %\caption{$\nu = 10$}
        %\label{fig:sim_glasso}
    \end{subfigure}
    
    \hfill % Adds flexible space between images
    \caption{The figure shows the oracle correlation matrix and the oracle precision matrix for the hidden factor set-up, when choosing $t=1000$, $k=300$, $r=6$ and each group has an exposure to exactly two distinct risk factors, to which the other groups do not have an exposure.}
    \label{fig:Sim_Oracle_Correlation}

\end{figure}

While GLasso and TLasso are build on a sequence of one repeated lambda parameter, GSlope and TSlope require to define a decreasing sequence of tuning parameters, i.e. $\boldsymbol{\lambda} = (\lambda_1, ..., \lambda_k)$. For the latter two, we follow \cite{Riccobello2025} and choose the sequence of lambda parameters according to Benjamini-Hochberg sequence for which the $j$-th tuning parameter is given as:\\

\begin{equation}\label{lambda_bh_slope}
    \lambda_{j}^{BH} = \frac{t_{\nu-2} ( 1-\frac{\alpha \times j}{2m} )}{\sqrt{\nu-2+t^{2}_{\nu-2} (1-\frac{\alpha \times j}{2m})}},\ \ \forall j=1, ..., k.
\end{equation}
 
As discussed in \cite{Sobczyk2019}, for $\alpha = 0.1$ this sequence targets a False Discovery Rate of $10\%$ for block-diagonal matrices. Since our aim here is to illustrate the potential of the different methods in terms of estimation error, rather than FDR control, we rescale the $\lambda$ sequences using the factor
\[
\gamma = \frac{\rho}{\bar{\boldsymbol{\lambda}}},
\]
where $\rho$ is a grid of 20 equally spaced values between 0 and 2.5, and $\bar{\boldsymbol{\lambda}}$ denotes the average of the $\lambda$ values obtained from (\ref{lambda_bh_slope}). We then choose the $\lambda$ sequence that minimizes the Frobenius norm distance between the oracle precision matrix and its estimate. The same procedure is used to select the tuning parameters for GLasso and TLasso, taking
$
\lambda = \gamma \lambda_1,
$
where $\lambda_1$ is the first element of the sequence in (\ref{lambda_bh_slope}).

In what follows, we compare the optimal performance of GLasso, TLasso, GSlope, and TSlope by simulating $m=100$ return time series from (\ref{HF_Time_Series}) and reporting the estimation error of these procedures for the corresponding optimal value of $\lambda$.
%estimating the precision matrix using the GLasso, TLasso, GSlope and TSlope procedures, for the set of 20 lambda values (sequences). For each of these 20 precision matrices we then choose as the optimal solution the result which has the minimum Frobenius norm distance between the oracle and the estimated precision matrix. This process is repeated $m=100$ times.

\begin{figure}[ht]
    \centering
    
    % --- First Subfigure ---
    \begin{subfigure}[b]{0.32\textwidth}
        \centering
        % Replace with your filename for nu=3
        \includegraphics[width=\linewidth]{Figures/Sim_Oracle_Precision.png}
        %\caption{$\nu = 5$}
        %\label{fig:sim_oracle_precision}
    \end{subfigure}
    \hfill % Adds flexible space between images
    % --- Second Subfigure ---
    \begin{subfigure}[b]{0.32\textwidth}
        \centering
        \includegraphics[width=\linewidth]{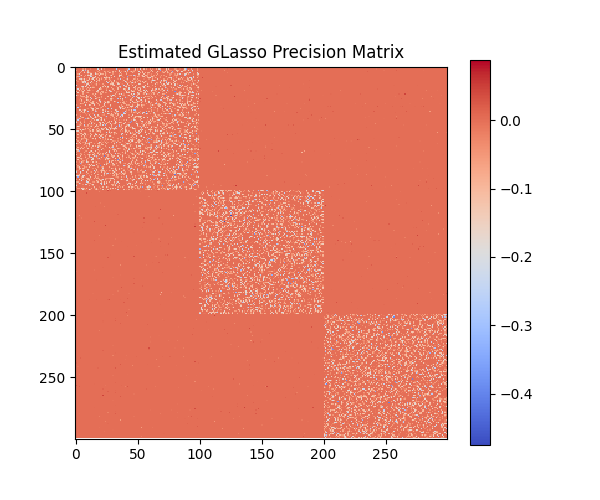}
        %\caption{$\nu = 10$}
        %\label{fig:sim_glasso}
    \end{subfigure}
    \hfill % Adds flexible space between images
    % --- Third Subfigure ---
    \begin{subfigure}[b]{0.32\textwidth}
        \centering
        % Replace with your filename for nu=10
        \includegraphics[width=\linewidth]{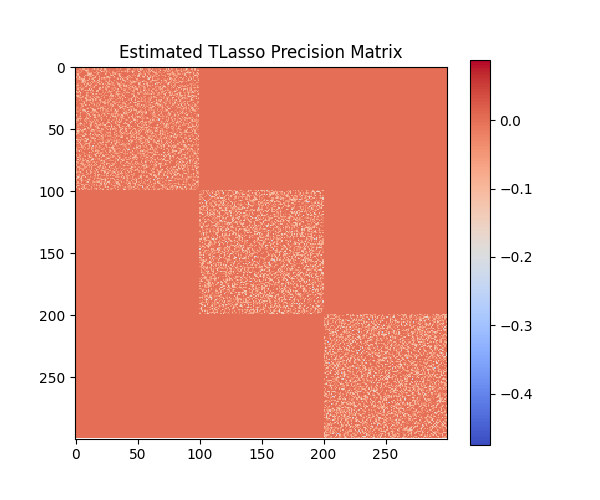}
        %\caption{$\nu = 1000$}
        %\label{fig:sim_tlasso}
    \end{subfigure}

      % --- Row 2 ---
    \begin{subfigure}{0.32\textwidth}
        \centering
        \includegraphics[width=\linewidth]{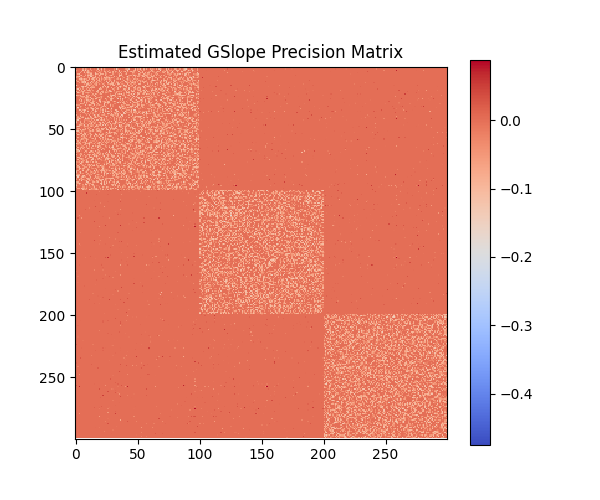}
        %\caption{Figure 4}
    \end{subfigure}
    \hfill
    \begin{subfigure}{0.32\textwidth}
        \centering
        \includegraphics[width=\linewidth]{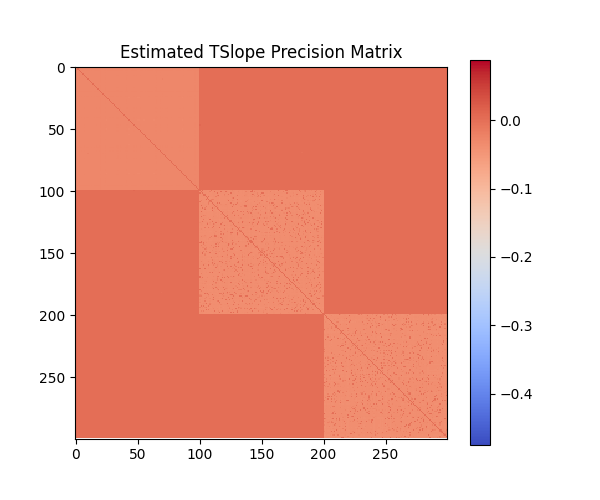}
        %\caption{Figure 5}
    \end{subfigure}
    \hfill
    \begin{subfigure}{0.32\textwidth}
        \centering
        \includegraphics[width=\linewidth]{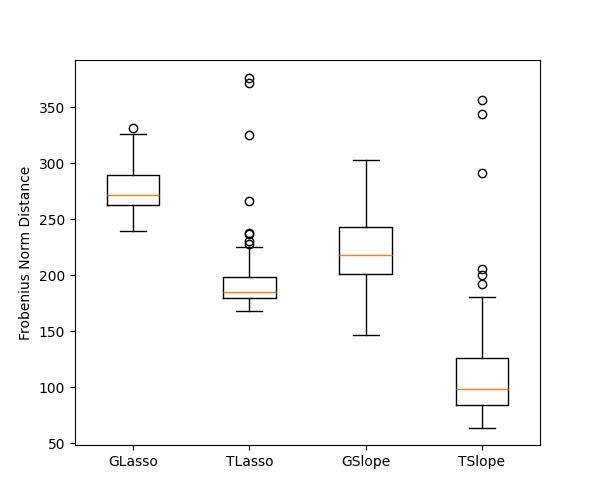}
        %\caption{Figure 6}
    \end{subfigure}
    
    \caption{The Figure shows from top left to bottom right: the oracle precision matrix, the estimated precision matrices of the GLasso, TLasso, GSlope and TSlope method, as well as the boxplots of the minimum Frobenius norm distances wrt. the orcale precision matrix across the $m=100$ simulation runs for the four methods, respectively. The matrices were selected based on the last simulation run and choosing the matrices across the grid of tuning parameters with the minimum Frobenius norm distance. Simulations have been performed by choosing $t=1000$, $k=300$ and $r=6$, and assuming that each group has an exposure to exactly two distinct risk factors to which the other groups do not have an exposure.}
    \label{fig:HIDDEN_FACTOR_RESULTS}
\end{figure}
Figure \ref{fig:HIDDEN_FACTOR_RESULTS} reports next to the estimated precision matrices for the GLasso, TLasso, GSlope and TSlope procedures, again the oracle precision matrix, as well as the boxplots of the minimum Frobenius norm distances for the respective procedures across the $m=100$ simulation runs. The reported matrices have been selected based on the last simulation run and choosing the matrix with the minimum Frobenius norm distance wrt. the oracle precision matrix, considering the grid of tuning parameters for the respective method.\\
The figure shows that all methods are able to resemble the underlying grouping structure. Still, looking closer at the results, we can see that TSlope performs best in extracting not only the dependence of assets across groups, but also the strength of dependence between assets in the same group and when comparing the estimated matrix with the oracle precision matrix. Furthermore, while for the other methods it appears that they capture the intergroup dependence among assets well, the dependence among assets within the same group can deviates significantly in magnitude from the oracle precision matrix. This deviation appears to be especially prone for the GLasso procedure. The superiority of the TSlope method is also confirmed by the boxplots from Figure \ref{fig:HIDDEN_FACTOR_RESULTS}, in which TSlope shows the lowest median Frobenius norm distance across the $m=100$ simulation runs and across all methods. Furthermore, we can observe that GLasso shows the worst median across all methods.
\newpage

\section{Empirical Evidence from Portfolios Sorted by Size and Book-to-Market}\label{sec:empirical}

%Intro and motivation
We complement the theoretical and simulation results presented in the previous sections with an empirical application aimed at assessing the clustering properties of the proposed estimators in a real-data setting. While the previous sections analyze the asymptotic behavior and pattern recovery properties under controlled data-generating processes, empirical data imply a more challenging exercise, as the underlying clustering structure is unknown and cannot be directly observed. To mitigate this difficulty, we exploit the presence of economically meaningful characteristics that induce a natural ordering among the variables. In particular, the empirical application considered in this section relies on data structured along two important dimensions: firm size and book-to-market ratio. These two characteristics provide a plausible proxy for latent similarity across variables. This setting allows us to investigate clustering behavior in the estimated precision matrices even in the absence of an observable ground-truth clustering.

%Brief data description
Specifically, we employ the data provided by Kenneth R. French's data library, and focus on the ``25 Portfolios Formed on Size and Book-to-Market'' dataset, which is widely used in the financial literature (see, e.g. \cite{Riccobello2025}) and provides a natural setting for investigating structured dependence patterns. This dataset consists of 25 portfolios, which are constructed at the end of each June as the intersections of five portfolios sorted on firm size (market equity, ME) and five portfolios sorted on the book-to-market ratio (BE/ME). Size breakpoints for year $t$ are defined using NYSE market equity quintiles at the end of June of year $t$, while BE/ME breakpoints are based on NYSE quintiles, where BE/ME for June of year $t$ is computed as book equity for the last fiscal year ending in $t-1$ divided by market equity in December of $t-1$. The portfolios include all NYSE, AMEX, and NASDAQ stocks for which market equity data are available for December of $t-1$ and June of $t$, together with positive book equity data for year $t-1$. Further details on the construction of the portfolios and data availability can be found in the Kenneth R. French Data Library (\url{https://mba.tuck.dartmouth.edu/pages/faculty/ken.french/}).

%Brief explorative analysis
In our empirical analysis, we use daily value-weighted returns, expressed in percentage points (i.e., multiplied by 100). Our sample spans the period from January 3, 2000 to October 31, 2025, for a total of 6,498 trading days. We reorder the 25 portfolios according to the BE/ME dimension. Therefore, portfolios sharing similar BE/ME characteristics are placed closer to each other, highlighting potential block structures in the dependence matrix. The resulting ordering follows the natural progression from low to high BE/ME portfolios, while preserving the ME classification within each group.

As a first descriptive overview of the data, Figure \ref{fig:boxplot_portfolios} reports boxplots of the daily returns across the 25 portfolios. The figure highlights substantial heterogeneity across portfolios. In particular, in terms of volatility, the BIG.HiBM portfolio exhibits the largest dispersion, with a standard deviation equal to 1.801, whereas ME5.BM2 displays the lowest volatility, with a standard deviation of 1.187. Moreover, all portfolios show the presence of pronounced extreme observations, both negative and positive, which is reflected in high kurtosis values. Specifically, kurtosis ranges from a minimum of 6.695 for ME2.BM1 to a maximum of 17.935 for ME5.BM4. Finally, in terms of average returns, the distributions appear to be substantially centered around zero.

\begin{figure}[ht]
    \centering
    \includegraphics[width=\textwidth]{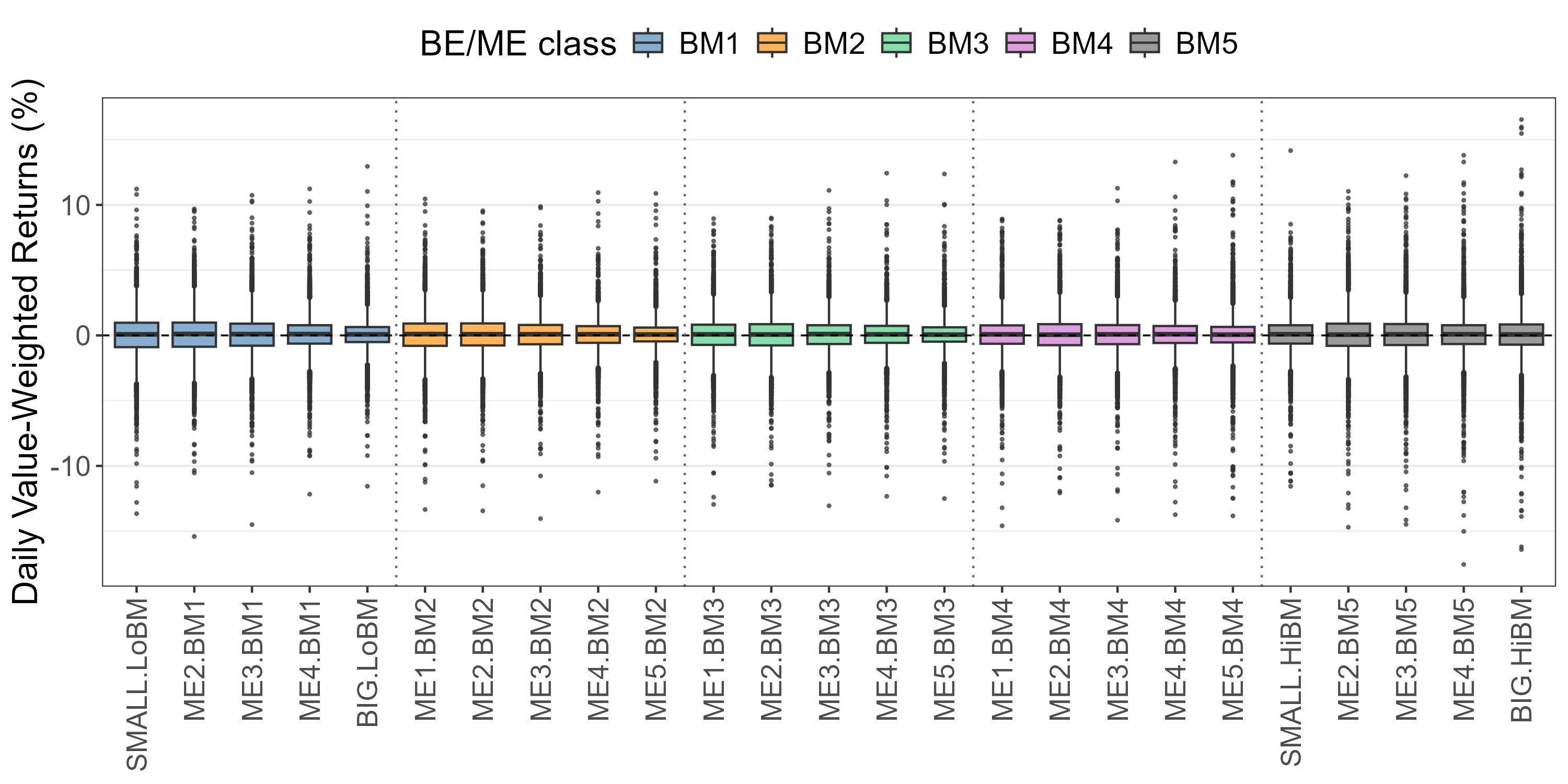}
    \caption{Boxplots of daily value-weighted returns for the 25 portfolios sorted by size and book-to-market (BE/ME) ratio computed from January 3, 2000 to October 31, 2025.}
    \label{fig:boxplot_portfolios}
\end{figure}

To further investigate the structure of the data, we analyze the time dynamics of portfolio returns. Figure \ref{fig:trend_portfolios} reports the evolution over time of the average daily returns for each book-to-market (BE/ME) class. Several stylized facts emerge from the figure. First, the series display clear evidence of volatility clustering \cite{Cont2001}, with relatively tranquil periods alternating with episodes of pronounced turbulence. This pattern is consistent with well-known features of financial time series and suggests the presence of regime-like dynamics in return variability. Periods of high volatility are clearly associated with well-known tail events, including the burst of the dot-com bubble in the early 2000s, the global financial crisis of 2007--2009, and the COVID-19 pandemic. In the final part of the sample, increased fluctuations due to trade-policy tensions and tariff-related uncertainty are also observed. Overall, the figure highlights strong common dynamics across BE/ME classes, while still allowing for heterogeneous behavior across groups, thereby motivating the subsequent analysis of dependence and clustering patterns in the estimated precision matrices.

\begin{figure}[ht]
    \centering
    \includegraphics[width=\textwidth]{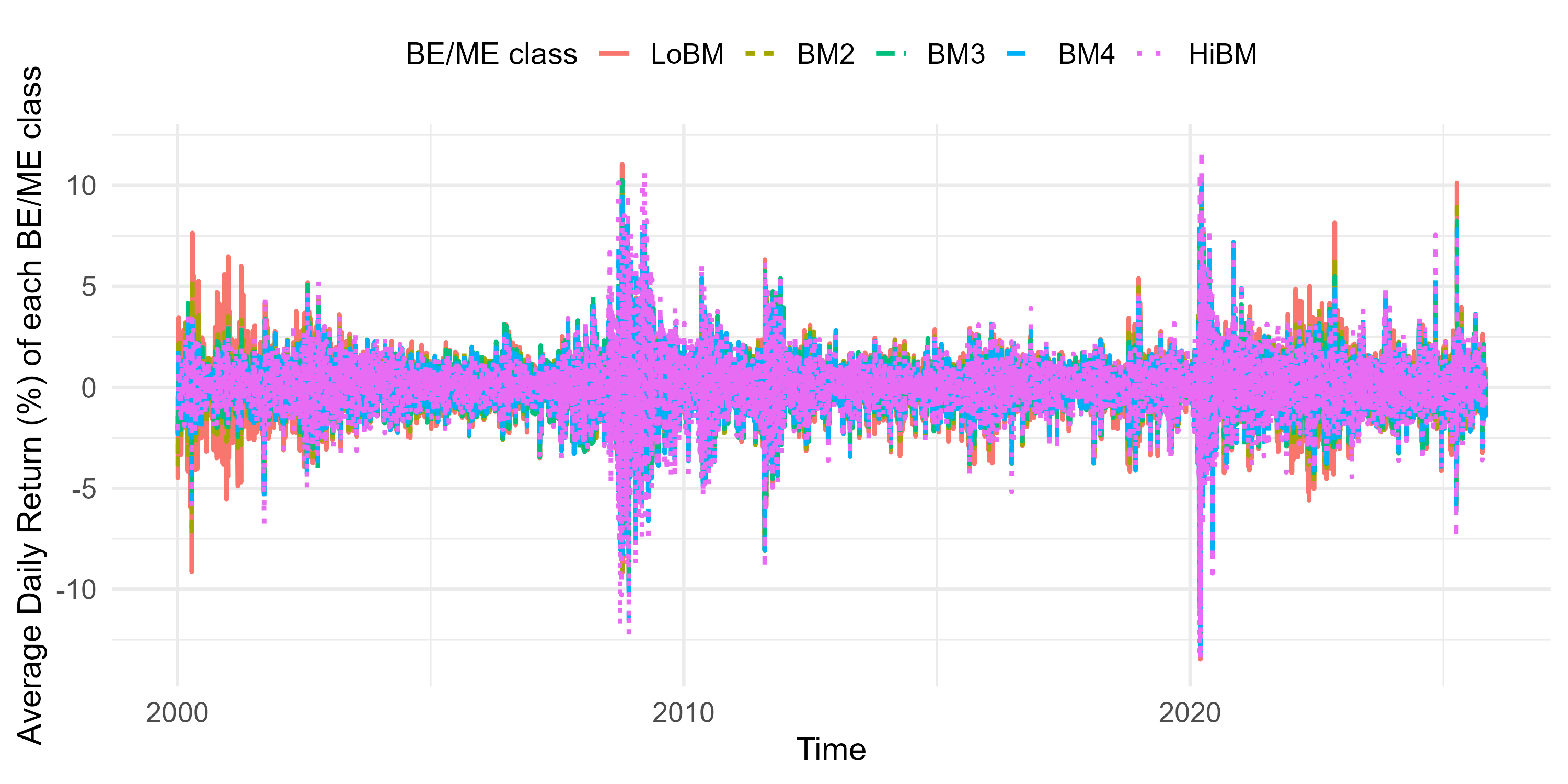}
    \caption{Time evolution of average daily value-weighted returns from January 3, 2000 to October 31, 2025; for each BE/ME class, the plotted series corresponds to the cross-sectional average of the five size-based portfolios.}
    \label{fig:trend_portfolios}
\end{figure}

After describing the main features of the data, we now turn to the empirical strategy adopted to investigate clustering patterns in the estimated precision matrices. In line with the methodology developed in the previous sections, clustering is performed directly on the off-diagonal entries of the precision matrix. While in the simulation study this approach allows a direct comparison with the true underlying parameters, in the empirical application the true clustering structure is not observable, and therefore clustering quality must be assessed using internal validation criteria.

To this end, we estimate GSLOPE and TSLOPE separately for each year in the sample period, from 2000 to 2025, obtaining a total of 26 precision matrix estimates for each model and each corresponding tuning-parameter sequence. Estimation on yearly subsamples allows us to capture potential time variation in dependence structures and provides multiple realizations that increase the stability of the clustering analysis. Before estimation, returns are standardized within each year. This step ensures comparability across subsamples and mitigates the effect of the pronounced volatility regimes highlighted in Figure \ref{fig:trend_portfolios}, where periods of relative stability alternate with episodes of elevated market turbulence. Standardization then prevents differences in scale from mechanically affecting the estimated dependence structure.

The penalization framework follows the one introduced in the simulation study. In particular, the overall level of regularization is controlled by a scalar parameter $\alpha$ multiplying a BH-type sequence. Differently from the simulation setting, where the BH($q$)-normalized sequence is based on Gaussian quantiles, in the empirical analysis we adopt a finite-sample version constructed from Student-$t$ quantiles. Let $p$ denote the number of variables and $m = p(p-1)/2$ the number of strictly upper-triangular entries of the precision matrix. In our empirical application, we have $p=25$ (i.e., the number of portfolios), which yields $m=300$ distinct off-diagonal entries. For $k=1,\ldots,m$, we then define:
$$
t_k = t_{n-2}^{-1} \left(1-\frac{q\, k}{2m}\right),
$$
and
$$
\lambda_k = \frac{t_k}{\sqrt{(n-2)+t_k^2}},
$$
where $t_{n-2}^{-1}(\cdot)$ denotes the quantile function of the Student-$t$ distribution with $n-2$ degrees of freedom; this construction yields a decreasing BH-type sequence of tuning parameters.\footnote{Throughout the empirical analysis we set $q=0.05$, in line with the conventional BH level used for two-sided testing.}

As in the simulation analysis, the global penalty strength is governed by the scalar parameter $\alpha$, which is explored over a logarithmically spaced grid from $10^{-4}$ to $4$, yielding 200 points along the regularization path. This wide range of $\alpha$ values allows us to investigate the behavior of the estimated precision matrices across different levels of regularization, from weak to strong shrinkage, and to assess the stability of the clustering structure across sparsity regimes.

For each model and for each tuning-parameter sequence, we retain the strictly upper-triangular part of the estimated precision matrix, excluding the main diagonal. Since $p=25$, this yields $m=300$ distinct edge coefficients for each yearly estimation. After vectoring these estimates across the 26 yearly subsamples, we construct a $300 \times 26$ matrix, where each row corresponds to a specific edge $(i,j)$, with $i<j$, and each column represents a yearly realization. Hence, each edge is characterized by a vector describing its temporal evolution over the sample period. Clustering is then performed by applying $k$-means to this matrix, grouping together edges that exhibit similar temporal patterns. We implement the $k$-means clustering method for different numbers of clusters and, for each of them, assess the clustering quality using the Calinski--Harabasz index \cite{Calinski1974}.\footnote{Preliminary analyses showed that performing clustering on edge vectors constructed across the 26 yearly estimations (i.e., using $300 \times 26$ matrices) leads to substantially more stable Calinski-Harabasz values along the regularization path compared with clustering based on a single cross-sectional estimate (i.e., a $300 \times 1$ vectors). This motivates the use of the temporal dimension to enhance the robustness of the clustering assessment.} This criterion balances within-cluster cohesion and between-cluster separation, with higher values indicating a more pronounced and well-separated clustering structure. Therefore, for each empirical setup, we compare the clustering performance of the proposed GSLOPE and TSLOPE estimators with the GLASSO benchmark. 

We stress the fact that GSLOPE and TSLOPE rely on a sequence of tuning parameters $\{\lambda_k\}$, whereas GLASSO is controlled by a single scalar regularization parameter. To ensure a meaningful comparison across methods, the GLASSO penalty level is calibrated using the average magnitude of the SLOPE tuning sequence. Specifically, we set the GLASSO tuning parameter equal to $\left(3.5 \; \alpha \,\overline{\lambda}\right)$, where $\overline{\lambda}$ denotes the average value of the tuning-parameter sequence defined for the SLOPE penalty function. The multiplicative factor $3.5$ is empirically chosen in order to align the overall magnitude of the Calinski-Harabasz indicator across methods, thereby allowing a more informative comparison of clustering patterns along the regularization path.

\begin{figure}[ht]
    \centering
    \includegraphics[width=\textwidth]{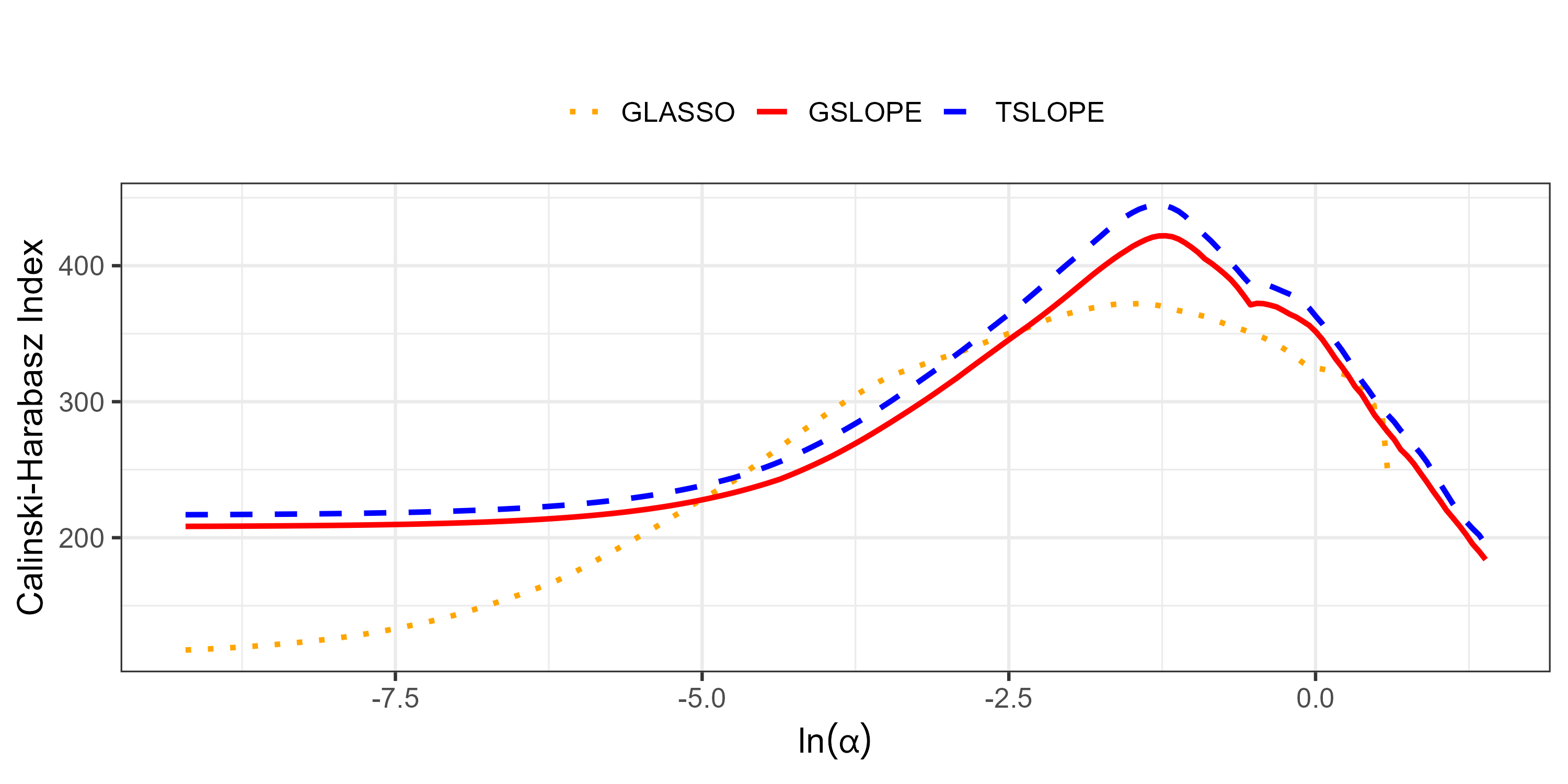}
    \caption{Calinski-Harabasz index as a function of $\ln(\alpha)$ for GSLOPE, TSLOPE, and GLASSO, reported for the case of three clusters.}
    \label{fig:CalinskiHarabasz3Clust}
\end{figure}

Figure \ref{fig:CalinskiHarabasz3Clust} reports the Calinski-Harabasz index as a function of $\ln(\alpha)$ for GSLOPE, TSLOPE, and GLASSO in the case of three clusters. Similar patterns are obtained by implementing the $k$-means method with different numbers of clusters.\footnote{For brevity, these additional results are not reported here but are available from the authors upon request.} We can see from Figure \ref{fig:CalinskiHarabasz3Clust} that TSLOPE systematically achieves the highest values of the Calinski--Harabasz index across the entire grid of $\alpha$ values, indicating a more pronounced and stable clustering structure. GSLOPE follows closely, while GLASSO displays substantially lower values throughout the regularization path. This ordering suggests that the structured penalization induced by SLOPE-based estimators leads to more clearly separated clusters compared with the standard sparsity-based benchmark. In particular, the Calinski--Harabasz index reaches its overall maximum under TSLOPE, with a value of 444.775 attained at $\alpha = 0.279$ (corresponding to $\ln(\alpha) = -1.277$ in Figure \ref{fig:CalinskiHarabasz3Clust}), highlighting the strongest clustering structure along the regularization path.

To assess whether the clustering results obtained from the Calinski-Harabasz analysis have an economically meaningful interpretation, we focus on the best-performing configuration while maintaining the analysis fixed at three clusters. Specifically, we consider the TSLOPE model estimated setting $\alpha = 0.279$, where the Calinski-Harabasz index reaches its maximum. Based on this specification, we extract the cluster allocation of the off-diagonal entries of the 26 precision matrices estimated for each year from 2000 to 2025, and visualize the resulting structure in Figure \ref{fig:PrecisionMatrHeatmap3Clust_Tslope_B2MSorted}. Figure \ref{fig:PrecisionMatrHeatmap3Clust_Tslope_B2MSorted} displays the cluster assignment of each entry $(i,j)$, with $i,j=1,\ldots,25$ and $i<j$, represented within a $25 \times 25$ matrix whose ordering follows the size and book-to-market structure described above. The three clusters are identified with three different colors: white (cluster 0), green (cluster 1) and orange (cluster 2). The label ``cluster 0'' is adopted purely for visualization purposes; as discussed below, this cluster predominantly corresponds to zero-valued entries across the underlying 26 precision matrices, representing the absence of conditional dependence between portfolio pairs. Clusters 1 and 2, in contrast, capture non-zero dependence patterns, highlighting heterogeneous interaction structures across portfolios. The representation given in Figure \ref{fig:PrecisionMatrHeatmap3Clust_Tslope_B2MSorted} provides an intuitive way to assess whether the statistically identified clusters exhibit economically interpretable patterns. In particular, it allows us to visually inspect the presence of localized dependence patterns that may be related to the economic characteristics underlying the portfolio construction.

\begin{figure}[ht]
    \centering
    \includegraphics[width=\textwidth]{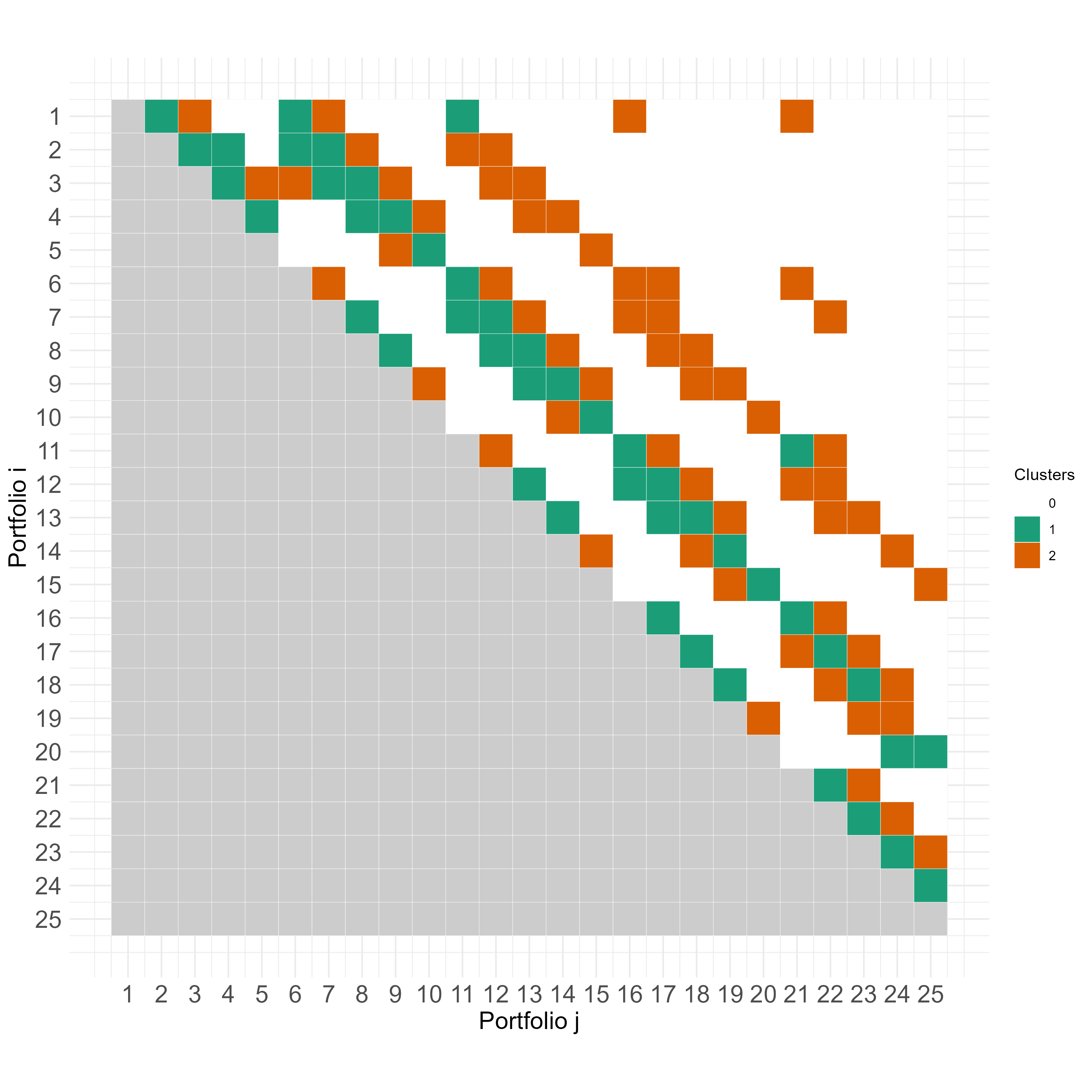}
    \caption{Heatmap of cluster assignments for the off-diagonal entries of the TSLOPE estimated precision matrices corresponding to the value of $\alpha$ that maximizes the Calinski-Harabasz index for the case of three clusters.}
    \label{fig:PrecisionMatrHeatmap3Clust_Tslope_B2MSorted}
\end{figure}

Figure \ref{fig:PrecisionMatrHeatmap3Clust_Tslope_B2MSorted} reveals a clear spatial organization of the identified clusters. Cluster 0, represented in white, is predominantly located farther away from the main diagonal. As anticipated, this cluster mainly corresponds to zero-valued entries of the underlying 26 precision matrices, indicating the absence of conditional dependence between portfolios that are relatively distant in terms of their economic characteristics. In contrast, cluster 1 tends to concentrate closer to the main diagonal. As will be shown below, this cluster is associated with the largest entries, in absolute value, of the estimated precision matrices, corresponding to stronger conditional dependence relationships. Given the ordering of portfolios according to size and book-to-market characteristics, this pattern suggests that stronger dependencies arise primarily among portfolios sharing similar economic features. In other words, proximity in the ordering dimension appears to be associated with higher similarity in dependence structure. Overall, Figure \ref{fig:PrecisionMatrHeatmap3Clust_Tslope_B2MSorted} highlights a localized dependence pattern, where economically similar portfolios exhibit stronger interactions, while more distant portfolios tend to be conditionally independent.

\begin{figure}[ht]
    \centering
    \includegraphics[width=\textwidth]{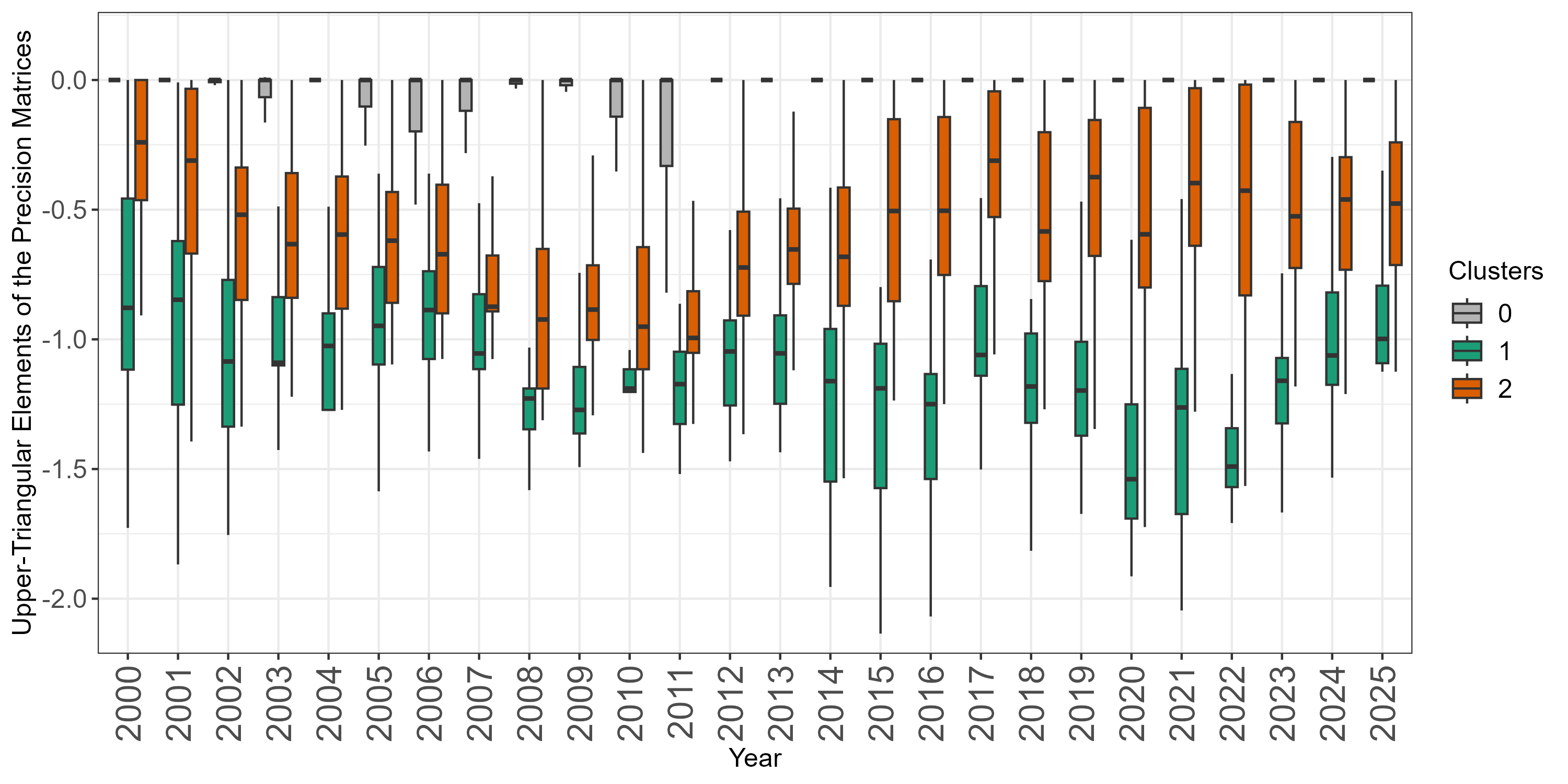}
    \caption{Boxplots of the off-diagonal entries (upper-triangular part) of the yearly precision matrices estimated with TSLOPE ($\alpha = 0.279$), reported for each year from 2000 to 2025 and grouped according to the cluster assignments obtained from the $k$-means procedure (with the number of clusters equal to three); outliers are omitted for visual clarity.}
    \label{fig:PrecisionMatrBoxplots3Clust_Tslope_B2MSorted}
\end{figure}

To further support the interpretation suggested by Figure \ref{fig:PrecisionMatrHeatmap3Clust_Tslope_B2MSorted}, Figure \ref{fig:PrecisionMatrBoxplots3Clust_Tslope_B2MSorted} reports boxplots of the off-diagonal entries (upper-triangular part) of the 26 yearly precision matrices estimated using TSLOPE with $\alpha = 0.279$, corresponding to the best-performing configuration identified in Figure \ref{fig:CalinskiHarabasz3Clust}. For visual clarity, outliers are omitted from the boxplots in Figure \ref{fig:PrecisionMatrBoxplots3Clust_Tslope_B2MSorted} in order to facilitate the comparison of the central distributional patterns across clusters and years. The boxplots are shown separately for each year from 2000 to 2025 and grouped according to the cluster assignment obtained from the clustering analysis. A clear pattern emerges from the figure. Across all years, the strongest relationships in absolute value are systematically associated with cluster 1, confirming that this cluster captures the most pronounced conditional dependence structure. As discussed above, these links mainly involve portfolios that are closer in terms of size and book-to-market characteristics. Interestingly, the strongest relationships are observed during the COVID-19 pandemic (year 2020). Cluster 2 displays intermediate magnitudes, while cluster 0 remains largely concentrated around zero, consistently with its interpretation as representing weak or absent conditional dependence.

\begin{figure}[ht]
    \centering
    \includegraphics[width=\textwidth]{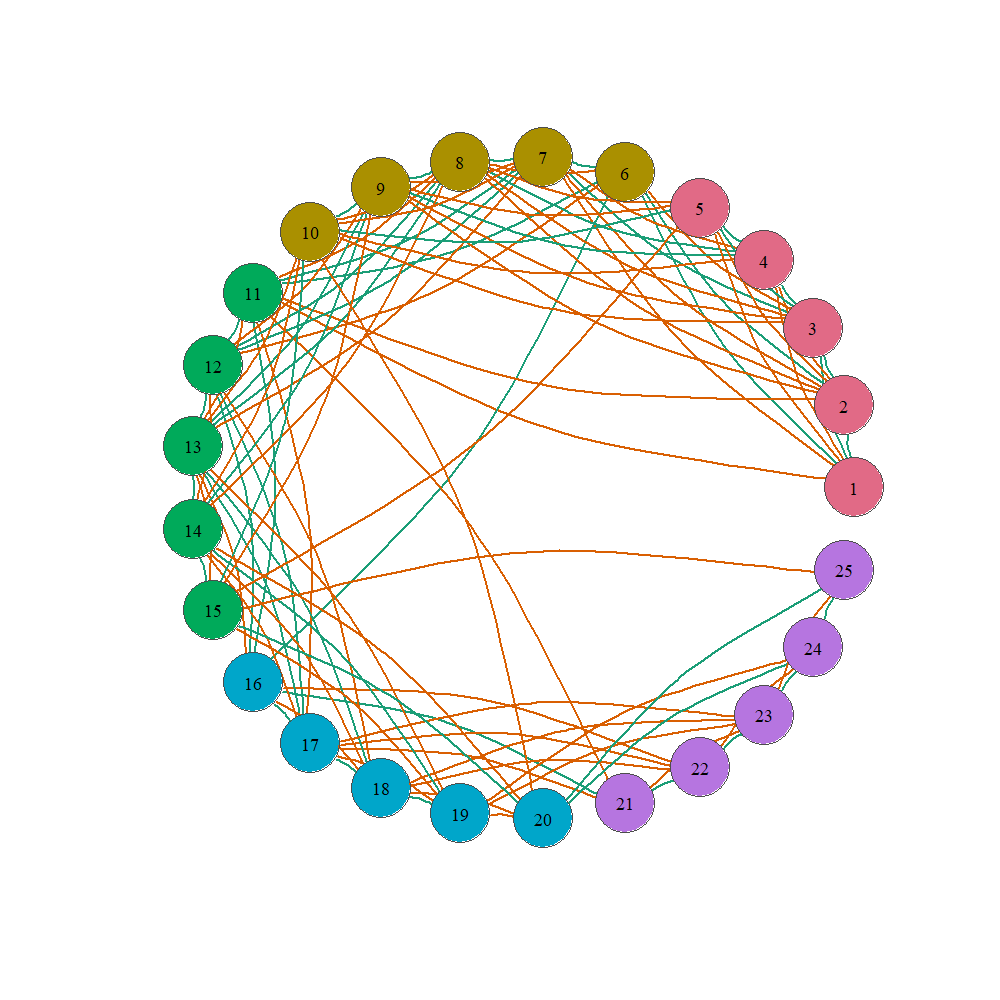}
    \caption{Network representation of the 25 portfolios based on the TSLOPE estimation at $\alpha = 0.279$ (number of clusters equal to three); nodes are colored according to their BE/ME class, and link colors correspond to the cluster assignments identified in Figure \ref{fig:PrecisionMatrHeatmap3Clust_Tslope_B2MSorted}.}
    \label{fig:Network3Clust_Tslope_B2MSorted}
\end{figure}

As a final evidence supporting the previous findings, Figure \ref{fig:Network3Clust_Tslope_B2MSorted} provides a network representation of the 25 portfolios based on the clustering structure obtained from the TSLOPE specification with $\alpha = 0.279$. Nodes are arranged on a circular layout and colored according to their BE/ME class, allowing a direct comparison between the economic grouping of portfolios and the dependence structure implied by the estimated precision matrix. Link colors follow the same cluster coding used in Figure \ref{fig:PrecisionMatrHeatmap3Clust_Tslope_B2MSorted}: green links correspond to cluster 1, associated with stronger conditional dependence, while orange links correspond to cluster 2, capturing intermediate dependence levels. Links belonging to cluster 0 are displayed in white, reflecting near-zero relationships. The resulting network visualization highlights a clear structural pattern. The central area appears relatively sparse, whereas stronger connections tend to concentrate along the periphery of the network, linking portfolios that are closer in terms of their economic characteristics. This evidence is consistent with the previous results and further supports the presence of localized dependence structures among economically similar portfolios.

Overall, the empirical analysis presented in this section provides consistent evidence in favor of the proposed SLOPE-based estimators, and in particular TSLOPE, in capturing meaningful clustering structures in the precision matrix. The combination of internal validation criteria, heatmap visualization, distributional analysis, and network representation highlights a clear pattern of localized dependence, where stronger conditional relationships emerge primarily among portfolios sharing similar economic characteristics. These findings suggest that the structured penalization induced by SLOPE not only improves statistical clustering performance but also reveals economically interpretable dependence structures in real financial data.

\section{Conclusion}\label{sec:conclusion}

This paper develops an asymptotic theory for Graphical SLOPE in precision matrix estimation, with particular emphasis on recovering structured dependence patterns beyond entrywise sparsity. In a fixed-$p$ asymptotic framework, we derive the limiting distribution of the rescaled estimation error, $\sqrt{n}(\widehat{\Theta}_n-\Theta_0)$, and of the asymptotic clustering pattern, for GSLOPE and TSLOPE and when the data are generated from multivariate elliptical distributions. A comparison of the asymptotic distributions highlights the advantage of gSLOPE over gLASSO when the precision matrix contains many equal entries. It also shows the advantage of TSLOPE over GSLOPE when the data are heavy-tailed. These results clarify the interaction between loss specification, tail behavior, and structured regularization, and show that asymptotic performance depends not only on sparsity but also on the extent to which the loss is adapted to the data-generating distribution.

These qualitative findings are further supported by simulation results in a high-dimensional setting and by an empirical application to portfolios sorted by size and book-to-market. In that setting, SLOPE-based estimators, especially TSLOPE, reveal stable and economically interpretable clustered dependence structures that are less apparent under standard sparsity-based methods. This suggests that structured penalties may be useful not only for improving statistical efficiency, but also for enhancing the interpretability of estimated financial networks.

Overall, the paper provides a theoretical and empirical foundation for using Graphical SLOPE and TSLOPE as tools for sparse and clustered precision matrix estimation. Several directions remain open for future research. A natural next step is to extend the present low-dimensional asymptotic analysis to high-dimensional regimes in which both $p$ and $n$ grow jointly. It would also be of interest to study analogous pattern-recovery questions for partial-correlation-based formulations and for other structured penalties in graphical models. More broadly, our results suggest that clustering-aware asymptotic theory may provide a useful framework for understanding how modern regularization methods recover interpretable dependence structures in multivariate data.

\section*{Acknowledgments}
IH, MB and JW acknowledge the support of the Swedish Research Council, grant no. 2020-05081. Sandra Paterlini would like to thank the EU COST Action CA21163 (HiTEc) ‘‘Text, functional and other high-dimensional data in econometrics: New models, methods, applications’’.

\newpage
\bibliographystyle{plain}
\bibliography{citation}

\newpage
\section{Appendix}\label{sec:appendix}
\subsection{Pattern, Pattern Space and Clustering Error of SLOPE}
\begin{example}\label{example: pattern and pattern space}
We illustrate the notion of a SLOPE pattern, the pattern space on the following example. Consider a precision matrix $\Theta_0$ and its lower-triangular vectorization $\vec_+(\Theta_0)$:
\begin{align*}
\Theta_0 &=
\left(\begin{matrix}
1 & 0.6 & 0.6 & 0.4 \\
\bl{0.6} & 1 & 0.6 & -0.4 \\
\bl{0.6} & \bl{0.6} & 1 & 0 \\
\re{0.4} & \re{-0.4} & \gr{0} & 1
\end{matrix}\right),
\quad
\vec_+(\Theta_0)=(\bl{0.6}, \bl{0.6}, \re{0.4}, \bl{0.6}, \re{-0.4}, \gr{0})
\end{align*}
The SLOPE pattern is given by
\begin{equation*}
    \mathbf{patt}(\vec_+(\Theta_0))=(\bl{2}, \bl{2}, \re{1}, \bl{2}, \re{-1}, \gr{0}),
\end{equation*}
and the pattern space $V_0:=\mathrm{span}\{\theta\in\mathbb{R}^{p(p-1)/2}: \mathbf{patt}(\theta)=\mathbf{patt}(\vec_+(\Theta_0))\}$ by
\begin{align*}
    V_0 = \{(\bl{a},\bl{a},\re{b},\bl{a},\re{-b},\gr{0}): a,b\in\mathbb{R}\}
\end{align*}

\end{example}\label{example: clustering error}
\begin{example}
    Consider $\Theta_0$ from the previous example. Suppose the rescaled error $\widehat{U}_n = \sqrt{n}(\widehat{\Theta}_n-\Theta_0)$ is
    \begin{align*}
\widehat{\Theta}_n =
\left(\begin{matrix}
1.1 & 0.65 & 0.68 & 0.37 \\
\bl{0.65} & 0.95 & 0.68 & -0.37 \\
\bl{0.68} & \bl{0.68} & 1.05 & 0.04 \\
\re{0.37} & \re{-0.37} & \gr{0.04} & 0.9
\end{matrix}\right)
\quad
\quad
\widehat{U}_n =
\begin{pmatrix}
1 & 0.5 & 0.8 & -0.3 \\
\bl{0.5} & -0.5 & 0.8 & 0.3 \\
\bl{0.8} & \bl{0.8} & 0.5 & 0.4 \\
\re{-0.3} & \re{0.3} & \gr{0.4} & -1
\end{pmatrix}
    \end{align*}
The lower-triangular vectorization $\hat{u}_n= \vec(\widehat{U}_n)$ decomposes into
\begin{align*}
    \hat{u}_n= (\bl{0.5},\bl{0.8}, \re{-0.3}, \bl{0.8}, \re{0.3}, \gr{0.4})=\underbrace{(\bl{0.7},\bl{0.7}, \re{-0.3}, \bl{0.7}, \re{0.3}, \gr{0})}_{P\hat{u}_n\in V_0}+\underbrace{(\bl{-0.2},\bl{0.1}, \re{0}, \bl{0.1}, \re{0}, \gr{0.4})}_{(I-P)\hat{u}_n\in V_0^{\perp}},
\end{align*}
where $P$ denotes the projection matrix onto $V_0$. The total error is 
\begin{equation*}
    \Vert \hat{u}_n\Vert_2^2 = \bl{0.5}^2 + \bl{0.8}^2 + (\re{-0.3})^2 + \bl{0.8}^2 + \re{0.3}^2 + \gr{0.4}^2 = 1.87,
\end{equation*}
of which the clustering error is:
\begin{equation*}
    \Vert (I-P)\hat{u}_n\Vert_2^2 = (\bl{-0.2})^2 + \bl{0.1}^2 + \bl{0.1}^2+ \gr{0.4}^2 = 0.22.
\end{equation*}
In this example, $0.22/1.87=11.8\%$ of the total estimation error are due to clustering error, the rest $88.2\%$ are due to the bias error.

\end{example}

\subsection{ Directional Derivative of Graphical SLOPE}\label{appendix: directional derivative}
For the graphical SLOPE norm with $\lambda_1>\dots>\lambda_{p(p-1)/2}>0$, 
\begin{equation*}
    \operatorname{Pen}(\Theta) = \sum_{j=1}^{p(p-1)/2} \lambda_j |\vec_+(\Theta)|_{(j)},
\end{equation*}

we find the explicit expression for the directional SLOPE derivative 
\begin{equation*}
\operatorname{Pen}'(\Theta_0;U)=\lim_{\varepsilon\to 0}\frac{1}{\varepsilon}(\operatorname{Pen}(\Theta_0+\varepsilon U)-\operatorname{Pen}(\Theta_0)).
\end{equation*}
Let $\theta^0=\vec_+(\Theta_0)$ and $u=\vec_+(U)$ be the strictly subdiagonal vectorizations in $\mathbb{R}^{p(p-1)/2}$. The directional derivative is
\begin{align*}
\operatorname{Pen}'(\Theta_0; U)&=\sum\limits_{j=1}^{p(p-1)/2}\lambda_{\pi(j)}\left[u_{j} \mathrm{sgn}(\theta^0_{j})\mathbb{I}[\theta^0_{j}\neq0]+\vert u_{j}\vert\mathbb{I}[\theta^0_{j}=0]\right],
\end{align*}
where $\pi$ is a permutation of size $p(p-1)/2$, which sorts $\theta^0+\varepsilon u$ by its absolute magnitudes;
\begin{align*}
    |\theta^0+\varepsilon u|_{\pi^{-1}(1)}\geq\dots \geq|\theta^0+\varepsilon u|_{\pi^{-1}(p(p-1)/2)}
\end{align*}
as $\varepsilon\searrow 0$.

%\subsection{Proofs}
\subsection{Proofs for the Gaussian Loss}\label{appendix: proofs for Gaussian Loss}

\begin{proof}[Proof of Theorem~\ref{theorem convergence in distribution}] 
    The statement follows from Theorem~\ref{main conditions theorem}. 
    Conditions (i)--(v) have been verified in Theorem 3 of \cite{bogdan2025identifying}. For completeness, we explicitly verify them here.
    
    First, observe that the loss function
    \begin{equation*}
        \ell(X,\Theta) = -\frac{1}{2}\log\det(\Theta) + \frac{1}{2}\operatorname{tr}(\Theta XX^T)
    \end{equation*}
    is a smooth map on the parameter space $\operatorname{Sym}_+(p)$ for every fixed $X\in\mathbb{R}^p$. The derivatives are given by
    \begin{align*}
        \nabla_\Theta \ell(X, \Theta) &= \frac{1}{2}(-\Theta^{-1} + XX^T),\\
        \nabla^2_\Theta \ell(X, \Theta) &= \frac{1}{2}(\Theta^{-1} \otimes \Theta^{-1}).
    \end{align*}
    Furthermore, the risk function takes the explicit form
    \begin{equation*}
        G(\Theta)=\mathbb{E}[\ell(X,\Theta)] = -\frac{1}{2}\log\det(\Theta) + \frac{1}{2}\operatorname{tr}(\Theta\Sigma),
    \end{equation*}
    where $\Sigma=\mathbb{E}[XX^T]$. 
    We fix an open neighborhood $B$ of $\Theta_0=\Sigma^{-1}$ in $\operatorname{Sym}_+(p)$ such that there exist constants $r_+, r_->0$ ensuring that all matrices in $B$ have eigenvalues uniformly bounded:
    \begin{equation}\label{compact neighborhood}
        r_- \leq \inf_{\Theta\in B} \lambda_{\min}(\Theta)\leq \sup_{\Theta\in B}\lambda_{\max}(\Theta)\leq r_+.
    \end{equation}
    This implies that the closure of $B$ is a compact subset of $\operatorname{Sym}_+(p)$.
    We now verify conditions (i)--(v) from Theorem~\ref{main conditions theorem}:

    \textbf{Condition (i)}: Note that the Hessian $\nabla^2_\Theta \ell(X, \Theta) = \frac{1}{2}(\Theta^{-1}\otimes \Theta^{-1})$ does not depend on $X$. The continuity of the inverse map and the compactness of the closure of $B$ imply that $\sup_{\Theta\in B}\lVert\nabla^2_\Theta \ell(X, \Theta)\rVert \leq M$ for some constant $M>0$.
    
    \textbf{Condition (ii)}: The function $G(\Theta)$ is smooth (and thus $C^3$) on $B$. The Hessian $C=\frac{1}{2}(\Theta_0^{-1}\otimes \Theta_0^{-1})$ is positive definite because $\Theta_0^{-1}=\Sigma$ is positive definite.

    \textbf{Condition (iii)}: This is satisfied because the expected gradient vanishes at the truth:
    \begin{equation*}
        \mathbb{E}[\nabla_\Theta \ell(X, \Theta_0)] = \frac{1}{2}(-\Theta_0^{-1}+\mathbb{E}[XX^T]) = \frac{1}{2}(-\Sigma + \Sigma) = 0.
    \end{equation*}
    Furthermore, $C_{\triangle}=\operatorname{Cov}(\vec(XX^T))/4$ is finite due to the finiteness of the fourth moment $\mathbb{E}[\lVert X\rVert^4]<\infty$.

    \textbf{Condition (iv)}: To establish uniform tightness, observe that the unpenalized minimization in \eqref{main objective} yields the MLE, $\widehat{\Theta}_n^{\text{MLE}}=(\frac{1}{n}\sum_{i=1}^nX^{(i)}X^{(i)^T})^{-1}$, which converges to the true precision $\Theta_0=\Sigma^{-1}$ by the Law of Large Numbers. The penalization term $n^{-1/2}\operatorname{Pen}(\Theta)$ is of order $O(n^{-1/2})$ and becomes negligible asymptotically. Consequently, the penalized estimator satisfies $\widehat{\Theta}_n\xrightarrow{p}\Theta_0$ as $n\to\infty$. The sequence is consistent, hence also uniformly tight.

    \textbf{Condition (v)}: A uniform envelope bound is obtained by applying the Cauchy--Schwarz inequality:
    \begin{equation*}
        \operatorname{tr}(\Theta X X^T) \le \lVert\Theta\rVert_F \,\lVert X X^T\rVert_F.
    \end{equation*}
    Using the continuity of $\log\det(\Theta)$ and the Frobenius norm $\lVert\Theta\rVert_F$, combined with the compactness of the closure of $B$, there exist constants $c_1, c_2$ such that $\sup_{\Theta\in B}|\ell(X,\Theta)|\leq c_1 + c_2\lVert XX^T\rVert_F$. This envelope is integrable due to the finiteness of the second moments.

    Thus, the loss satisfies all regularity conditions required in Corollary~3.4 of \cite{hejny2025asymptotic}, and Theorem~\ref{theorem convergence in distribution} follows.
\end{proof}

For the next proposition and subsequent results, we introduce the following notation:
\begin{notation}
For any matrices $M_1, M_2 \in \mathbb{R}^{p^2\times p^2}$, we write 
$$M_1\equiv M_2,$$ 
whenever $M_1 \vec(A) = M_2\vec(A)$ for every $A\in\operatorname{Sym}(p)$. In other words, $M_1\equiv M_2$ if the linear maps coincide when restricted to the subspace of vectorized symmetric matrices.
\end{notation}

\begin{proposition} 
Let $u$ be a random vector uniformly distributed on the unit sphere $\mathbb{S}^{p-1}$. Then
\begin{align}\label{spherical moments}
\mathbb{E}[uu^T\otimes uu^T]\equiv\frac{1}{p(p+2)}\left(2I_{p^2}+\vec(I)\vec(I)^T\right),
\end{align}
in the sense of the notation defined above.
\end{proposition}

\begin{proof} 
We begin with the general formula for the fourth moments of the spherical distribution:
\begin{align*}
\mathbb{E}[uu^T\otimes uu^T] =\frac{1}{p(p+2)}(I_{p^2}+K_p+\vec(I)\vec(I)^T),
\end{align*}
where $K_{p}$ is the commutation matrix characterized by $K_{p}\vec(A)=\vec(A^T)$ for every matrix $A$. This identity can be derived by expressing $u=Z/\lVert Z\rVert$, where $Z \sim \mathcal{N}_p(0, I_p)$. Since $u$ is independent of the radius $\lVert Z\rVert$, we have
\begin{equation*}
    \mathbb{E}[uu^T\otimes uu^T] = \frac{\mathbb{E}[ZZ^T\otimes ZZ^T]}{\mathbb{E}[\lVert Z\rVert^4]}.
\end{equation*}
Using Isserlis' theorem, the numerator expands to terms involving Kronecker deltas, which corresponds to $I_{p^2} + K_p + \vec(I)\vec(I)^T$ in matrix notation. The denominator is given by $\mathbb{E}[\lVert Z\rVert^4]=p(p+2)$ (see, e.g., Subsection 2.6.2 in \cite{anderson1958introduction}).

Finally, observe that for any symmetric matrix $A \in \operatorname{Sym}(p)$, $K_p \vec(A)=\vec(A^T)=\vec(A)$. Thus, $K_p \equiv I_{p^2}$ on the subspace of symmetric matrices, and the result follows.
\end{proof}

\begin{proof}[Proof of Proposition~\ref{proposition for elliptical distributions}]
    It remains to establish the representation for the covariance in \eqref{covariance for elliptical distributions}. From Theorem~\ref{theorem convergence in distribution}, we have
    \begin{align*}
        C_{\triangle}&= \frac{1}{4}\Cov(\vec(XX^T))\\
        &=\frac{1}{4}\left(\mathbb{E}\left[XX^T\otimes XX^T\right] - \vec(\Theta_0^{-1})\vec(\Theta_0^{-1})^T \right)\\
        &= \frac{1}{4}(\Theta_0^{-1/2}\otimes\Theta_0^{-1/2})\left(\mathbb{E}[R^4]\mathbb{E}\left[uu^T\otimes uu^T\right] - \vec(I)\vec(I)^T \right) (\Theta_0^{-1/2}\otimes \Theta_0^{-1/2}).
    \end{align*}
    For the special case of the Gaussian distribution, where $R_0^2\sim\chi_p^2$, we have the fourth moment $\mathbb{E}[R_0^4]=p(p+2)$, and the identity $C=C_{\triangle}$.
    We can express the general covariance by adding and subtracting $C$ expressed as $C_{\triangle}$ above with $R=R_0$. Note that the constant term involving $\vec(I)\vec(I)^T$ cancels out in the difference:
    \begin{align*}
        C_{\triangle} &= C + (C_{\triangle}-C)\\
        & = C + \frac{1}{4}\left(\mathbb{E}[R^4]-p(p+2)\right) (\Theta_0^{-1/2}\otimes\Theta_0^{-1/2})\left(\mathbb{E}\left[uu^T\otimes uu^T\right] \right) (\Theta_0^{-1/2}\otimes \Theta_0^{-1/2})\\
        & \equiv C + \frac{1}{4}\left(\dfrac{\mathbb{E}[R^4]}{p(p+2)}-1\right) (\Theta_0^{-1/2}\otimes\Theta_0^{-1/2})\left(2I_{p^2}+\vec(I)\vec(I)^T\right) (\Theta_0^{-1/2}\otimes \Theta_0^{-1/2})\\
        & = C + \left(\dfrac{\mathbb{E}[R^4]}{p(p+2)}-1\right) \left(C + \frac{1}{4} \vec(\Theta_0^{-1})\vec(\Theta_0^{-1})^T\right),
    \end{align*}
    where we have used the identity \eqref{spherical moments} and the definition of $C$ in the final steps.
\end{proof}

\subsection{Proofs for the t-Loss}\label{appendix: proofs for t-Loss}

\begin{proof}[Proof of Lemma~\ref{lemma t-hessian,covariance}]
To derive the explicit expressions for the Hessian $C$ and the covariance $C_{\triangle}$, we consider the derivatives of the loss \eqref{t-distribution loss}:
\begin{align*}
\nabla_{\Theta}\ell(x, \Theta) &= - \dfrac{1}{2}\Theta^{-1}+\dfrac{\nu+p}{2}\dfrac{1}{\nu+x^{T}\Theta x}xx^T,\\
\nabla_{\Theta}^2\ell(x, \Theta) &= \dfrac{1}{2} (\Theta^{-1}\otimes\Theta^{-1}) - \dfrac{\nu+p}{2}\dfrac{1}{(\nu+x^{T}\Theta x)^2}\vec(xx^T)\vec(xx^T)^T.
\end{align*}

\textbf{Step 1:} (Derivation of the Hessian $C$).
From the representation $X=\Theta_0^{-1/2}uR$, we obtain:
\begin{equation*}
    X^T\Theta_0X=R^2 \hspace{2cm}\text{and}\hspace{2cm} \frac{XX^T}{X^T\Theta_0X}=\Theta_0^{-1/2}uu^T\Theta_0^{-1/2}.
\end{equation*}
Substituting these into the Hessian formula and taking expectations yields
\begin{align*}
    \mathbb{E}[\nabla_{\Theta}^2\ell(X, \Theta_0)] 
    &= \dfrac{1}{2} (\Theta_0^{-1}\otimes\Theta_0^{-1}) - \mathbb{E}\left[\frac{\nu+p}{2}\dfrac{(X^{T}\Theta_0 X)^2}{(\nu+X^{T}\Theta_0 X)^2}\left(\dfrac{\vec(XX^T)}{X^T\Theta_0 X}\vec\left(\dfrac{XX^T}{X^T\Theta_0 X}\right)^T\right)\right]\\
    & = \dfrac{1}{2} (\Theta_0^{-1/2} \otimes \Theta_0^{-1/2})\left(I\otimes I - \mathbb{E}[\xi_R]\cdot \mathbb{E}[uu^T\otimes uu^T]\right) (\Theta_0^{-1/2} \otimes \Theta_0^{-1/2}),
\end{align*}
where $\xi_{R}=(\nu+p)R^4/(\nu+R^2)^2$. Using identity~\eqref{spherical moments} for $\mathbb{E}[uu^T\otimes uu^T]$ and simplifying leads directly to the expression for $C$ in \eqref{t-hessian}.

\textbf{Step 2:} (Derivation of the covariance $C_{\triangle}$).
First, we verify the consistency condition. From $X=\Theta_0^{-1/2}uR$, we have:
\begin{equation*}
    \mathbb{E}\left[\frac{XX^T}{\nu+X^T\Theta_0X}\right] = \mathbb{E}\left[\frac{R^2}{\nu+R^2}\Theta_0^{-1/2}uu^T\Theta_0^{-1/2}\right]=\mathbb{E}\left[\frac{R^2}{p(\nu+R^2)}\right]\Theta_0^{-1}.
\end{equation*}
Therefore, the first-order condition holds:
\begin{equation}\label{assumption verification}
    \mathbb{E}[\nabla_{\Theta}\ell(X, \Theta_0)]=0 \iff \mathbb{E}\left[\frac{(\nu+p)R^2}{p(\nu+R^2)}\right]=1,
\end{equation}
which matches the assumption \eqref{key assumption on R for t-distribution}.
We now compute the covariance of the score vector:
\begin{align*}
    C_{\triangle} &:= \operatorname{Cov}(\vec(\nabla_{\Theta}\ell(X, \Theta_0)))\\
      &= \operatorname{Cov}\left(\dfrac{\nu+p}{2}\dfrac{X^{T}\Theta_0 X}{\nu+X^{T}\Theta_0 X} \vec\left(\dfrac{XX^T}{X^{T}\Theta_0 X}\right)\right)\\
      &= \frac{1}{4}\operatorname{Cov}\left(\dfrac{(\nu+p)R^2}{\nu+R^2}\vec\left(\Theta_0^{-1/2}uu^{T}\Theta_0^{-1/2}\right)\right).
\end{align*}
Using the definition $\operatorname{Cov}(Y) = \mathbb{E}[YY^T] - \mathbb{E}[Y]\mathbb{E}[Y]^T$ and the fact that $\mathbb{E}[uu^T]=p^{-1}I_p$, we get
\begin{align*}
      C_{\triangle} &= \frac{1}{4}\left(\mathbb{E}\left[\dfrac{(\nu+p)^2R^4}{(\nu+R^2)^2}\right]\mathbb{E}\left[\vec(\Theta_0^{-1/2}uu^{T}\Theta_0^{-1/2})\vec(\Theta_0^{-1/2}uu^{T}\Theta_0^{-1/2})^T\right] - \vec(\Theta_0^{-1})\vec(\Theta_0^{-1})^T \right)\\
      &= \frac{1}{4}\left(\Theta_0^{-1/2} \otimes \Theta_0^{-1/2}\right)\left[ (\nu+p)\mathbb{E}[\xi_R]\mathbb{E}\left[uu^T\otimes uu^T\right]-\vec(I)\vec(I)^T\right] \left(\Theta_0^{-1/2} \otimes \Theta_0^{-1/2}\right).
\end{align*}
Using \eqref{spherical moments} again, we can relate this to $C$. The difference $C_{\triangle}-C$ simplifies to
\begin{align*}
    \left(\Theta_0^{-1/2} \otimes \Theta_0^{-1/2}\right) \left(\left(\frac{\nu+p+2}{p(p+2)}\mathbb{E}[\xi_R]-1\right)\left(\frac{1}{2}I_{p^2}+\frac{1}{4}\vec(I)\vec(I)^T\right)\right)\left(\Theta_0^{-1/2} \otimes \Theta_0^{-1/2}\right),
\end{align*}
which matches the second term in \eqref{t-covariance}.

\textbf{Step 3:} (Positive definiteness of $C$).
To verify that $C$ is positive definite, we write $C=(\Sigma^{1/2}\otimes \Sigma^{1/2}) Q (\Sigma^{1/2}\otimes \Sigma^{1/2})$, with
\begin{equation*}
    Q:=\frac{1}{2}I_{p^2} - \frac{\mathbb{E}[\xi_R]}{p(p+2)}\left( I_{p^2} + \frac{1}{2}vv^T \right),
\end{equation*}
where $v=\vec(I_p)$. It suffices to show $Q \succ 0$.
Note that
\begin{align}\label{bounds on xi_R}
    0<\mathbb{E}[\xi_R]=\mathbb{E}\left[\frac{R^2}{\nu+R^2}\cdot \frac{(\nu+p)R^2}{\nu+R^2}\right]<\mathbb{E}\left[ \frac{(\nu+p)R^2}{\nu+R^2}\right]=p.
\end{align}
Any vector $w \in \mathbb{R}^{p^2}$ orthogonal to $v$ is an eigenvector of $Q$ with eigenvalue
\begin{equation*}
    \lambda_1 = \frac{1}{2}-\frac{\mathbb{E}[\xi_R]}{p(p+2)} > \frac{1}{2}-\frac{p}{p(p+2)} = \frac{1}{2} - \frac{1}{p+2} > 0.
\end{equation*}
The vector $v$ itself is an eigenvector with eigenvalue
\begin{equation*}
    \lambda_2 = \frac{1}{2}-\frac{\mathbb{E}[\xi_R]}{p(p+2)}\left(1+\frac{\lVert v\rVert^2}{2}\right) = \frac{1}{2}-\frac{\mathbb{E}[\xi_R]}{p(p+2)}\left(1+\frac{p}{2}\right) = \frac{1}{2}-\frac{\mathbb{E}[\xi_R]}{2p}.
\end{equation*}
Using the bound $\mathbb{E}[\xi_R] < p$, we see that $\lambda_2 > \frac{1}{2} - \frac{1}{2} = 0$.
Thus, all eigenvalues of $Q$ are positive, implying $C$ is positive definite.
\end{proof}

\begin{proof} [Proof of Theorem~\ref{theorem convergence in distribution for t distribution}]
To establish asymptotic convergence, we verify the conditions in Theorem~\ref{main conditions theorem}. The risk function is given by
\begin{equation*}
    G(\Theta)=\mathbb{E}[\ell(X,\Theta)]=-\frac{1}{2}\log\det(\Theta)+\frac{\nu + p}{2} \mathbb{E}[\log\big(\nu + X^\top \Theta X\big)].
\end{equation*}
Note that $G(\Theta)$ is minimized at $\Theta_0$. The first-order condition $\nabla_\Theta G(\Theta_0)=0$ holds due to the assumption \eqref{key assumption on R for t-distribution} on $R$ and the verification in \eqref{assumption verification}. 

We fix an open neighborhood $B$ around $\Theta_0$ in $\operatorname{Sym}_+(p)$, such that there exist positive constants $r_+, r_->0$ ensuring that all matrices in $B$ have eigenvalues uniformly bounded:
\begin{equation*}
    r_- \leq \inf_{\Theta\in B} \lambda_{\min}(\Theta)\leq \sup_{\Theta\in B}\lambda_{\max}(\Theta)\leq r_+.
\end{equation*}
This implies that the closure of $B$ is a compact subset of $\operatorname{Sym}_+(p)$. We now verify conditions (i)--(v) of Theorem~\ref{main conditions theorem}.

\textbf{Condition (i):} Let $\preceq$ denote the Loewner partial order on symmetric matrices, where $A_1\preceq A_2$ if $A_2-A_1$ is positive semi-definite. Recall the Hessian derived in Lemma~\ref{lemma t-hessian,covariance}:
\begin{equation*}
    \nabla_{\Theta}^2\ell(X, \Theta) = \dfrac{1}{2} (\Theta^{-1}\otimes\Theta^{-1}) - \dfrac{\nu+p}{2}\dfrac{1}{(\nu+X^{T}\Theta X)^2}\vec(XX^T)\vec(XX^T)^T.
\end{equation*}
The second term is the subtraction of a positive semi-definite matrix. Therefore,
\begin{equation*}
    \nabla^2_\Theta \ell(X, \Theta) \preceq \frac{1}{2}(\Theta^{-1}\otimes\Theta^{-1}).
\end{equation*}
Using the operator norm (largest absolute eigenvalue), we have
\begin{equation*}
    \lVert\nabla^2_\Theta \ell(X, \Theta)\rVert \leq \frac{1}{2}\lVert\Theta^{-1}\otimes\Theta^{-1}\rVert = \frac{1}{2}\lVert\Theta^{-1}\rVert^2 = \frac{1}{2\lambda_{\min}(\Theta)^2}.
\end{equation*}
Given the bounds on $B$, this is bounded by $M = 1/(2r_-^2)$, satisfying condition (i).

\textbf{Condition (ii):} The function $G(\Theta)$ is smooth on $B$, satisfying the $C^3$ (thrice continuous differentiability) requirement. The positive definiteness of the Hessian $C$ was established in Lemma~\ref{lemma t-hessian,covariance}.

\textbf{Condition (iii):} The gradient condition $\mathbb{E}[\nabla_\Theta \ell(X, \Theta_0)]=0$ holds by assumption \eqref{key assumption on R for t-distribution} and the derivation in \eqref{assumption verification}. Furthermore, $C_{\triangle} < \infty$ is a consequence of the bound $0 < \mathbb{E}[\xi_R] < p$ derived in \eqref{bounds on xi_R}.

\textbf{Condition (iv):} The uniform tightness of the sequence $\widehat{\Theta}_n$ is established in Lemma~\ref{lemma: uniform tightness}.

\textbf{Condition (v):} We establish a uniform envelope over the compact closure of $B$. For any $\Theta \in B$, we have $r_- \lVert X\rVert_2^2 \leq X^T\Theta X \leq r_+ \lVert X\rVert_2^2$. 
The term $\log(\nu + X^T\Theta X)$ is bounded from below by $\log(\nu)$ and from above by $\log(\nu + r_+\lVert X\rVert_2^2)$.
Consequently,
\begin{align*}
    |\ell(X,\Theta)| &\leq \frac{1}{2}|\log\det(\Theta)| + \frac{\nu+p}{2}|\log(\nu+X^T\Theta X)|\\
    &\leq \frac{p}{2}\max(|\log r_-|, |\log r_+|) + \frac{\nu+p}{2}\max\left(|\log(\nu)|, \log(\nu+ r_+\lVert X\rVert_2^2)\right).
\end{align*}
This provides an integrable envelope. Specifically, for an elliptical vector $X = \Theta_0^{-1/2}uR$, the expectation $\mathbb{E}[\log(\nu + r_+\lVert X\rVert_2^2)]$ behaves asymptotically like $\mathbb{E}[\log(R^2)]$, which is finite by the assumption $\mathbb{E}[\log(1+R^2)] < \infty$.

Thus, all regularity conditions are satisfied, and the conclusion of Theorem~\ref{theorem convergence in distribution} follows.
\end{proof}

\begin{lemma}\label{lemma: uniform tightness}
    The sequence $(\widehat{\Theta}_n)_{n\in\mathbb{N}}$ is uniformly tight; i.e., for every $\varepsilon>0$ there exists a compact set $K\subset\operatorname{Sym}_+(p)$ such that 
    \begin{equation*}
        \mathbb{P}(\widehat{\Theta}_n\in K)\geq1-\varepsilon,
    \end{equation*} 
    for every $n\in\mathbb{N}$.
\end{lemma}

\begin{proof}
We consider a compact set $K$ of the form
\begin{equation*}
    K=\{\Theta\in\operatorname{Sym}_+(p): r_-\leq \lambda_{\min}(\Theta)\leq\lambda_{\max}(\Theta)\leq r_+\},
\end{equation*}
for some strictly positive constants $r_+, r_- > 0$. To ensure that $\widehat{\Theta}_n$ remains within such a set $K$ with high probability, it suffices to show the coercivity of the objective function. This means verifying that the objective loss tends to infinity as $\lambda_{\max}(\Theta)\to \infty$ or $\lambda_{\min}(\Theta)\to 0$.

Consider the spectral decomposition $\Theta = \sum_{j=1}^p \lambda_j v_j v_j^T$, where $0 < \lambda_p \leq \dots \leq \lambda_1$ are the eigenvalues and $\{v_1, \dots, v_p\}$ is an orthonormal basis. Let $\lambda_{\max}(\Theta) = \lambda_1 = \lVert\Theta\rVert$. For any fixed $\delta > 0$, we have
\begin{align*}
    \log(\nu + X^T\Theta X) &\geq \log(\nu + \lambda_1 (v_1^T X)^2)\\
    &\geq \log(\nu + \lambda_1 (v_1^T X)^2) \mathbb{I}{\{|v_1^TX|\geq \delta\}}\\
    &\geq \log(\lambda_1 (v_1^T X)^2) \mathbb{I}{\{|v_1^TX|\geq \delta\}}\\
    &= \left(\log(\lambda_1) + \log((v_1^TX)^2) \right)\mathbb{I}{\{|v_1^TX|\geq \delta\}}.
\end{align*}
Using the inequality $\log\det(\Theta) \leq p \log(\lambda_{\max}(\Theta))$, the empirical loss $\mathcal{L}_n(\Theta) = \frac{1}{n}\sum_{i=1}^n \ell(X_i, \Theta)$ can be bounded from below:
\begin{align*}
    \mathcal{L}_n(\Theta) &= -\frac{1}{2}\log\det(\Theta) + \frac{\nu+p}{2}\frac{1}{n}\sum_{i=1}^n \log(\nu + X_i^T\Theta X_i)\\
    &\geq -\frac{p}{2}\log(\lambda_1) + \frac{\nu+p}{2}\frac{1}{n}\sum_{i=1}^n \left(\log(\lambda_1)+\log((v_1^T X_i)^2) \right)\mathbb{I}{\{|v_1^T X_i|\geq \delta\}}\\
    &\geq \left(\frac{\nu+p}{2}S_n - \frac{p}{2}\right)\log(\lambda_{\max}(\Theta)) +  \frac{\nu+p}{2}\log(\delta^2) S_n,
\end{align*}
where 
\begin{equation*}
    S_n := \frac{1}{n}\sum_{i=1}^n \mathbb{I}{\{|v_1^T X_i|\geq \delta\}}.
\end{equation*}
Since the logarithmic term involving $\delta$ is bounded for fixed $\delta$, the dominant term is the coefficient of $\log(\lambda_{\max}(\Theta))$. We analyze $S_n$ using uniform convergence. The class of indicator functions $\{\mathbb{I}{\{|v^T x|\geq \delta\}} : v\in\mathbb{S}^{p-1}\}$ has finite VC-dimension (as intersections of half-spaces), implying it is a Glivenko-Cantelli class (see Example 19.17 in \cite{VanDerVaart1998}). 
Define $f_n(v):=\frac{1}{n}\sum_{i=1}^n \mathbb{I}{\{|v^T X_i|\geq \delta\}}$ and $f(v):=\mathbb{P}(|v^T X|\geq \delta)$. For any fixed $\delta\geq 0$, the uniform law of large numbers implies $\sup_{v\in\mathbb{S}^{p-1}}|f_n(v)-f(v)|\xrightarrow{p} 0$. 
Since the infimum is a continuous functional with respect to the uniform norm, we obtain
\begin{equation*}
    S_n \geq \inf_{v\in\mathbb{S}^{p-1}} f_n(v) \xrightarrow{p} \inf_{v\in\mathbb{S}^{p-1}} f(v) = \inf_{v\in\mathbb{S}^{p-1}} \mathbb{P}(|v^T X|\geq \delta).
\end{equation*}

It remains to bound the limit from below. Since $X$ has a non-singular covariance, $\mathbb{P}(|v^T X| \geq \delta) \to 1$ as $\delta \to 0$ for every fixed $v$. Moreover, the mapping $\delta \mapsto \mathbb{P}(|v^T X| \geq \delta)$ is monotone. By Dini's Theorem, this convergence is uniform on the compact sphere $\mathbb{S}^{p-1}$. 
Thus, $\inf_{v \in \mathbb{S}^{p-1}} \mathbb{P}(|v^T X| \geq \delta) \to 1$ as $\delta \to 0$.
Consequently, we can choose $\delta > 0$ sufficiently small such that
\begin{equation*}
    \inf_{v\in\mathbb{S}^{p-1}} \mathbb{P}(|v^T X|\geq \delta) > \frac{p+1}{\nu+p}.
\end{equation*}
Then, with probability approaching 1, $S_n \geq \frac{p+1}{\nu+p}$, and consequently:

\begin{equation*}
    \mathcal{L}_n(\Theta) \geq \frac{1}{2}\log(\lambda_{\max}(\Theta)) - O_p(1).
\end{equation*}
The loss tends to infinity as $\lVert\Theta\rVert \to \infty$. As a result, the empirical loss is bounded from below uniformly for large $\lVert\Theta\rVert$ and $n$. Precisely, there exists $N\in\mathbb{N}$ such that:
\begin{equation*}
    \inf_{\lVert\Theta\rVert\geq r_+} \inf_{n\geq N}\mathcal{L}_n(\Theta) \to \infty,
\end{equation*}
as $r_+\to\infty$. The same uniform lower bound holds for the penalized loss, as the penalty is non-negative. For the finitely many terms $n < N$, tightness follows from the fact that any finite collection of random variables is tight.
Consequently, for any $\varepsilon>0$, we can choose $r_+>0$ large enough such that 
\begin{equation*}
    \mathbb{P}\left(\lVert\widehat{\Theta}_n \rVert \leq r_+\right)\geq 1-\varepsilon \quad \forall n \in \mathbb{N}.
\end{equation*}

To establish uniform tightness of $\widehat{\Theta}_n$, it remains to bound the smallest eigenvalue $\lambda_{\min}(\widehat{\Theta}_n)$ away from zero. We condition on the event $\{\lVert\widehat{\Theta}_n\rVert \leq r_+\}$.
Using the inequality $\log\det(\Theta) \leq \log(\lambda_{\min}(\Theta)) + (p-1)\log(\lambda_{\max}(\Theta))$ and observing that the term $\log(\nu + X^T\Theta X)$ is bounded from below by $\log(\nu)$, we have for any $\lVert\Theta\rVert\leq r_+$:
\begin{align*}
    \mathcal{L}_n(\Theta) &\geq -\frac{1}{2} \log(\lambda_{\min}(\Theta)) - \frac{p-1}{2}\log(r_+) + \frac{\nu+p}{2}\log(\nu).%\\
    %&= -\frac{1}{2} \log(\lambda_{\min}(\Theta)) + O(1).
\end{align*}
As $\lambda_{\min}(\Theta)\to 0$, the loss tends to infinity. Thus, minimizing the loss requires the smallest eigenvalue to be bounded away from zero. 
We conclude that for every $\varepsilon>0$, there exist constants $r_+,r_->0$ such that
\begin{equation*}
    \mathbb{P}\left(r_-\leq \lambda_{\min}(\widehat{\Theta}_n )\leq\lambda_{\max}(\widehat{\Theta}_n )\leq r_+\right)\geq 1-\varepsilon \quad \forall n \in \mathbb{N},
\end{equation*}
which proves that the sequence $\widehat{\Theta}_n$ is uniformly tight.
\end{proof}

\end{document}